\documentclass{cpamart1}     
\startingpage{1}                      

\authorheadline{C. Sormani and I. Stavrov}
\titleheadline{Geometrostatic Manifolds}

\usepackage{amsmath, amsthm, amssymb}
\usepackage{graphicx}
\usepackage{mathbbol}
\usepackage{txfonts}
\usepackage[usenames]{color}
\usepackage{enumerate,mdwlist}
\usepackage{tikz}
\usepackage{alltt}
\usetikzlibrary{shapes}
\usetikzlibrary{plotmarks}

\usepackage{mathptmx}

\newtheorem{theorem}{Theorem}[section]
\newtheorem{lemma}[theorem]{Lemma}

\theoremstyle{definition}

\theoremstyle{remark}

\newtheorem{thm}[theorem]{Theorem}

\newtheorem{cor}[theorem]{Corollary}   
\newtheorem{lem}[theorem]{Lemma}   
\newtheorem{prop}[theorem]{Proposition}
\newtheorem{defn}[theorem]{Definition}
\newtheorem{conj}[theorem]{Conjecture}
\newtheorem{rmrk}[theorem]{Remark}   
\newtheorem{example}[theorem]{Example}

\newcommand{\be}{\begin{equation}}
\newcommand{\ee}{\end{equation}}

\newcommand{\e}{\varepsilon}

\newcommand{\diam}{\operatorname{diam}}

\newcommand{\area}{\operatorname{Area}}
\newcommand{\vol}{\operatorname{Vol}}

\begin{document}                        


\title{Geometrostatic Manifolds of small ADM Mass}

\author{Christina Sormani}{CUNY Graduate Center and Lehman College}
\author{Iva Stavrov Allen}{Lewis \& Clark College}





\begin{abstract}
We bound the locations of outermost minimal surfaces in 
geometrostatic manifolds whose ADM mass is small relative to the
separation between the black holes  
and prove the Intrinsic Flat Stability of the Positive Mass Theorem 
in this setting.
\end{abstract}

\maketitle   






\section{Introduction}

Geometrostatic manifolds are asymptotically Euclidean solutions of the time symmetric vacuum Einstein-Maxwell constraint equations
\be
R(g)=2|E|^2 \textrm{\ \ and\ \ } \mathrm{div} E=0.
\ee
They take the form of 
\be\label{M}
(M\,,\,g)=\left(\,\,{\mathbb{R}^3}\setminus\,
P \,,\, (\,\chi\,\psi)^2\, \delta\,\right)
\ee
where $P=\{p_1,...p_n\}$ is the ``set of holes'',
where $\delta$ is the Euclidean metric, and  where
$\psi>0,\, \chi>0$ on ${\mathbb{R}^3}\setminus P$
with $\Delta \chi=0,\, \Delta \psi=0$, and $\chi, \psi \to 1$ as $r\to \infty$ in $\mathbb{R}^3$.   These manifolds were studied by Brill and Lindquist \cite{Brill-Lindquist}, Misner \cite{Misner-geo} and
Lichnerowicz \cite{Lich-1944}. The conformal factors $\chi$ and $\psi$ are given by 
\be\label{chi-psi}
\chi(x)= 1+\sum_{i=1}^n \frac{\,\alpha_i\,}{\rho_i} \quad\textrm{ and }\quad
\psi(x)=1+\sum_{i=1}^n \frac{\,\beta_i\,}{\rho_i}
\ee
where $\alpha_i>0$ and $\beta_i>0$ are arbitrary and 
$\rho_i=\rho_i(x)=|p_i-x|$ is the Euclidean distance from $x$ to $p_i$.   This metric
on $M$ is complete and asympototically Euclidean as $x\to p_i$ and $|x|\to \infty$, so that
we have $n+1$ ends.   The
electric field, $E$, is the gradient of the electrostatic potential, $\ln(\psi/\chi)$, up to a sign.

We see from \cite{Brill-Lindquist} that the ADM mass of the $(n+1)^{st}$ end, where $|x|\to \infty$, is
\be\label{ADM}
m=m_{n+1}=\sum_{i=1}^n (\alpha_i+\beta_i), 
\ee
and that the $i^{th}$ end, where $x\to p_i$, has ADM mass
\be\label{m_i}
m_i=\alpha_i+\beta_i
+\sum_{j\neq i} \frac{(\beta_i \alpha_j +\beta_j\alpha_i)}{r_{i,j}}
\ee 
and charge
\be\label{q_i}
q_i=\beta_i-\alpha_i+\sum_{j\neq i}   \frac{(\beta_i \alpha_j -\beta_j\alpha_i)}{r_{i,j}}
\ee
with $r_{i,j}=|p_i-p_j|$ denoting the Euclidean distance from $p_i$ to $p_j$.  
We define the ``separation factor'' of the set of holes, $P$, to be
\be\label{separation}
\sigma=\sigma(P)=\min\big\{\sigma_1,..., \sigma_n, |p_1|,...,|p_n|\big\}
\textrm{ where } \sigma_i=\min\big\{r_{i,j}:\, j\neq i\big\}
\ee
When $q_i=0$ this manifold has zero scalar curvature and otherwise
these manifolds have nonnegative scalar curvature.

Note that
the Riemannian Schwarzschild black hole of mass $m_1$ is an 
example of a geometrostatic manifold with a single point $p_1=0$,
charge, $q_1=0$, and mass, $m_1=2\alpha_1=2\beta_1$,
\be\label{Sch}
\left(M,g_{\mathrm{Sch}}\right)=\left(\,{\mathbb{R}^3}\setminus\{p_1\},\,\, (\chi\psi)^2 \delta\,\right)
\quad
\textrm{ where }\quad 
\chi(x)=\psi(x)= 1+ \tfrac{m_1}{2\rho_1}
\ee
and where $\delta$ is the Euclidean tensor.  It has two asymptotically
flat ends: as $r=|x|\to \infty$ and as $r=|x|\to 0$.   Between the two ends is
a neck with a closed minimal surface 
$
\Sigma= \{x:\, |x|=m_1/2\}  
$
which is called an ``apparent horizon".   Note
that 
\be
\area_g(\Sigma)= 
\left(1+\tfrac{m_1}{2(m_1/2)}\right)^4 4\pi \left(\tfrac{m_1}{2}\right)^2
=16\pi m_1^2
\textrm{ and }m=m_{n+1}=m_1.
\ee
Any geometrostatic manifold,
$(M,g)$, satisfying (\ref{M})-(\ref{chi-psi}) may be viewed as a collection of
$n$ black holes each with mass, $m_i$, and charge, $q_i$.   

Recall the Positive Mass Theorem  
states that the ADM mass of an asymptotically Euclidean manifold with
nonnegative scalar curvature is nonnegative, and when the ADM mass
is 0 the manifold is Euclidean space \cite{Schoen-Yau-positive-mass}.   This is easily seen to hold in
the geometrostatic setting: by (\ref{ADM}), we have
\be
m=\sum_{i=1}^n (\alpha_i+\beta_i) \ge 0.
\ee
and
\be \label{static}
m=0 \quad \implies\quad  \alpha_i=\beta_i=0 \quad
\implies \quad (M,g)=({\mathbb{R}^3}, \delta).
\ee
In particular there are no black holes if the ADM mass is zero.  

The
Penrose inequality states that if
$M'$ is an asymptotically Euclidean manifold of nonnegative
scalar curvature 
whose boundary $\partial M'$ is an outermost minimal surface
then
\be \label{Penrose}
m_{ADM}(M',g) \ge \sqrt{ \frac{\area_g(\partial M')}{16\pi} \,}.
\ee  
This was proven for $M'$ with a connected
boundary by Huisken-Ilmanen in \cite{Huisken-Ilmanen}.
Bray \cite{Bray-Penrose} proved the inequality even when the boundary
has more than one connected component.

\begin{defn}\label{outermost-surface}
In a Brill-Lindquist geometrostatic manifold, $M$, the outermost minimal
surface $\Sigma_i$ about $p_i$ is a closed connected minimal surface, 
$\Sigma_i=\partial \Omega_i \subset {\mathbb{R}}^3$ where 
$p_i\in \Omega_i$, such that for any $\Omega\subset {\mathbb{R}}^3$
with $p_i \in \Omega$ and $\partial \Omega$ a closed minimal surface, we have
$\Omega \subseteq \Omega_i$.  
\end{defn}

In the geometrostatic setting, each end $x\to p_i$ has such 
an outermost minimal surface, $\Sigma_i$ (see Example~\ref{ex-min-apart}).  Such
minimal surfaces exist by the work of Huisken-Ilmanen \cite{Huisken-Ilmanen} which we review within Subsection~\ref{sect-HI}.  Note that it is possible when
$p_i$ and $p_j$ are close enough, that $\Sigma_i=\Sigma_j$
(see Example~\ref{ex-min-share}).  In fact such an example is computed numerically by Brill and Lindquist when $m_i$ and $m_j$ are large compared to $r_{i,j}$ \cite{Brill-Lindquist}.   In this paper we assume
separation factor, $\sigma$, as in (\ref{separation}),
is large compared to the ADM mass, $m$, of the manifold, $M$,
and conclude 
that each outermost minimal surface, $\Sigma_i$, is distinct and 
is located in an annular region
around $p_i$:

\begin{thm}\label{minsurf}
There exists a universal constant $C_1\gg 1$ such that if a geometrostatic
manifold, $M$, has
\be
m<\frac{\sigma}{\,20 C_1},
\ee
where $m$ is the ADM mass of the end at infinity as in (\ref{ADM}) and
$\sigma$ is the separation factor as in (\ref{separation}), then
for all $i$ the $i^{th}$ outermost minimal surface of $M$ satisfies
\be 
\Sigma_i \subseteq B_\delta\left(p_i, 2C_1\sqrt{\area_g(\Sigma_i)/\pi}\right)\setminus B_\delta\left(p_i, \frac{\alpha_i\beta_i}{4C_1(\alpha_i+\beta_i)}\right).
\ee 
In particular, since $2C_1\sqrt{\area_g(\Sigma_i)/\pi}\,\le \,8C_1m<\sigma/2$, 
the surfaces $\Sigma_i$ for distinct $i$ are disjoint. 
\end{thm}

This theorem is proven in Section~\ref{sect-minsurf}.

Huisken-Ilmanen defined the exterior region, which we call an outermost region,
in \cite{Huisken-Ilmanen}.  We review this notion carefully in 
Subsection~\ref{sect-HI} just stating the definition in our setting here:

\begin{defn}\label{outermost-region}
The ``outermost region", $M'\subset M$, is 
\be
 M'={\mathbb{R}}^3 \setminus \bigcup_{i=-n'}^n \Omega_i.
 \ee
 where each $\Omega_i$ is diffeomorphic to a ball and has outward minimizing boundary $\Sigma_i$.  For $i\in \{1,...,n\}$ these are the regions $\Omega_i$ containing $p_i$ that we defined above and for $i\le 0$ these are possible additional
 outward minimizing regions (which we conjecture do not exist).   Note 
 that $M'$ has one end as $|x|\to \infty$ and has
an outermost minimizing boundary, 
\be
\partial M'=\Sigma=\bigcup_{i=-n'}^n \Sigma_i,
\ee
and no closed interior minimal surfaces.  
\end{defn}

The outermost region
satisfies the time symmetric vacuum Einstein-Maxwell equation.  So it has 
nonnegative scalar curvature and, if the charge is $0$, then it has zero scalar curvature.  By the Penrose
Inequality in the form proven by Bray in \cite{Bray-Penrose}  
we have
\be 
m=m_{n+1}=m_{ADM}(M',g)\ge \sqrt{\frac{1}{16\pi}\sum_{i=-n'}^n
\area_g(\Sigma_i)\,}.
\ee
We prove the following theorem in Section~\ref{sect-minsurf}:
\begin{thm} \label{thm-M'}
Let $C_1$ be the constant of Theorem \ref{minsurf}, and suppose that $M$ is geometrostatic with $\sigma >20 mC_1$. For all $1\le i\le n$ there exists lengths $\gamma_i\le 8C_1m$ such that the outermost region $M'$ of Definition \ref{outermost-region} satisfies
\be
{\mathbb{R}}^3 \setminus \left(\bigcup_{i=1}^n B_\delta\left(p_i, \gamma_i\right)\right) \subseteq M' 
\subseteq {\mathbb{R}}^3 \setminus \left(\bigcup_{i=1}^n B_\delta\left(p_i, \frac{\alpha_i\beta_i}{4C_1(\alpha_i+\beta_i)}\right)\right)
\ee
\end{thm}
The reader may want to note that the definition and more detailed estimates involving $\gamma_i$ are presented in Sections \ref{sect-minsurf} and ~\ref{sect-almost-rigidity}, respectively. 

It has been conjectured that if a sequence of pointed asymptotically flat manifolds, 
$(M'_k, g_k, x_k)$,
with nonnegative scalar curvature whose boundaries are outermost
minimal surfaces, has $m_{ADM}(M'_k)\to 0$, then
$(M'_k,g_k)$ converge in the pointed intrinsic flat sense to
Euclidean space, $({\mathbb{R}^3}, \delta, 0)$, assuming the manifolds
are centered on well chosen points, $x_k$, which do not disappear
down increasingly deep wells.   When proposing this conjecture and proving it in the rotationally symmetric case, Lee and the second author demonstrated that this conjecture would be false if it were stated with a stronger notion of convergence \cite{LeeSormani1}.   
Lan-Hsuan Huang, Dan Lee and the first author have proven this conjecture in the graph setting assuming
additional hypotheses including one that requires all level sets to be outward minimizing  \cite{HLS-Crelle}.
It is unknown whether the setting considered in \cite{HLS-Crelle} can
include multiple black holes.  In Section~\ref{sect-almost-rigidity}, we prove this conjecture for geometrostatic manifolds:  

\begin{thm} \label{Almost-Rigidity}
Let $(M_k,g_k)$ be a sequence of geometrostatic manifolds with 
outermost regions, $M'_k$, such that
\be\label{hyp-thm1.5}
m_{ADM}(M'_k)\to 0 
\qquad \textrm{ and } \qquad
\frac{m_{ADM}(M'_k)}{\sigma(M_k)}\to 0 
\ee
where $\sigma(M_k)$
 is the separation factor of the set of holes
$P_k$ of $M_k$ as
in (\ref{separation}).  Assume furthermore that there is some 
$R_0\ge 0$ such that 
the set of accumulation points of
$\rho(\cup_k P_k) \,\cap \,(R_0,\infty)$ is of measure $0$ where
$\rho(x)=|x|$.

Then $(M_k',g_k)$ converges in the pointed intrinsic flat sense to
Euclidean space.   More precisely, for almost all $R>R_0$ 
the ball  $B_{g_k}(0,R) \subset (M_k',g_k)$ converges to the Euclidean ball 
$B_\delta(0,R) \subset \mathbb{E}^3$ in the intrinsic flat sense. 
\end{thm}

Before proving either theorem we present examples in Section~\ref{sect-examples} which illustrate why some aspects of the proofs are technically
difficult.   We review Huisken-Ilmanen and prove Theorem~\ref{minsurf} in Section~\ref{sect-minsurf}.  We review Intrinsic Flat Convergence 
particularly work of the first author with Lakzian \cite{Lakzian-Sormani}, and 
prove Theorem~\ref{Almost-Rigidity} in Section~\ref{sect-almost-rigidity}.  In the Appendix we provide additional information about geometrostatic manifolds needed within the paper.

%

The authors thank the organizers of the 2014 conference {\em Geometric Analysis and Relativity Conference} at the University of Science and Technology of China, at which this collaboration commenced.
We further thank the organizers of the 2016 workshop {\em Geometric Analysis and General Relativity} at the Banff International Research Station, during which the final version of this work was formulated.


 
\section{Examples}\label{sect-examples}

In this section we describe the location of the outermost region and outermost
minimal surfaces in a variety of Brill-Lindquist geometrostatic manifolds.

\begin{example}\label{ex-Sch}
The Riemannian Schwarzschild manifold as in (\ref{Sch}) can be depicted as in 
Figure~\ref{fig-Sch} to emphasize that it has two asymptotically flat ends: one as 
$|x|\to 0$ and one as $|x|\to \infty$.   The ends are not quite as flat as depicted here.  The outermost minimal surface,
$\Sigma=\{x:\, |x|=m_1/2\}$, lies in the neck between the two ends and the outermost region, $M'=\{x:\, |x|>m_1/2\}$, lies above $\Sigma$ in this image.
\end{example}

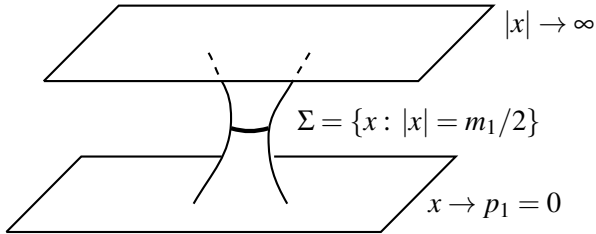
\begin{figure}[h] 
\centering
\begin{tikzpicture}[scale=.5]
\draw[thick] (-5,-5) to (5, -5) to (7,-3) to (-3, -3) to (-5, -5);
\draw[thick] (-6,-9) to (4, -9) to (6,-7) to (1.15, -7); 
\draw[thick] (-0.25,-7) to (-4, -7) to (-6, -9);
\draw[thick, dashed] (-.6, -4.25) to [out=-60, in=120] (-0.25, -5);
\draw[thick] (-.25, -5) to [out=-60, in=90] (0, -6) to  [out=-90, in=60](-1, -8.25);
\draw[thick, dashed] (2.1, -4.25) to [out=-120, in=60] (1.65, -5);
\draw[thick] (1.65, -5) to [out=-120, in=90] (1, -6.5) to  [out=-90, in=120](1.5, -8.25);
\draw[ultra thick] (0,-6.25) to [out=-15, in=-165] (1,-6.25);
\node[right] at (5,-8.3) {$x\to p_1=0$};
\node[right] at (1.5,-6.1) {$\Sigma=\{x:\, |x|=m_1/2\}$};
\node[right] at (7,-3.5) {$|x|\to\infty$};
\end{tikzpicture}
\caption{Example~\ref{ex-Sch}: The Riemannian Schwarzschild Manifold}
\label{fig-Sch}
\end{figure}

\begin{example}\label{ex-ERN}
The extreme Reissner-Nordstrom black hole has a metric of the form
(\ref{M})-(\ref{chi-psi}) with $n=1$, $\alpha_1>0$ and $\beta_1=0$.  See
Figure~\ref{fig-ERN}.  It has one end as $|x|\to \infty$ which is asymptotically
flat and one end as $x\to p_1$ which is asymptotically cylindrical.  It is not
a Brill-Lindquist geometrostatic manifold and has no outermost minimal surface
$\Sigma$.  One may view this example as having an infinitely long neck. 
\end{example}

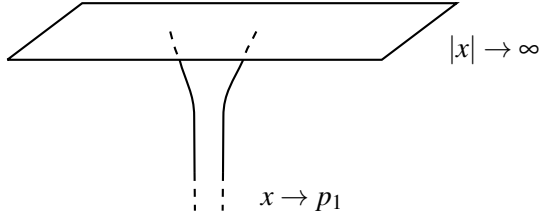
\begin{figure}[h]
\centering
\begin{tikzpicture}[scale=.5]
\draw[thick] (-5,-5) to (5, -5) to (7,-3.5) to (-3, -3.5) to (-5, -5);
\draw[thick, dashed] (-.65, -4.25) to [out=-75, in=115] (-0.4, -5);
\draw[thick] (-.4, -5) to [out=-75, in=90] (-0.01, -6.5) to (0,-8);
\draw[thick, dashed] (0,-8) to (0.01,-9);
\draw[thick, dashed] (1.65, -4.25) to [out=-115, in=75] (1.3, -5);
\draw[thick] (1.3, -5) to [out=-115, in=90] (0.77, -6.5) to (0.76,-8);
\draw[thick, dashed] (0.76,-8) to (0.75, -9);
\node[right] at (1.5,-8.75) {$x\to p_1$};
\node[right] at (6.5,-4.75) {$|x|\to\infty$};
\end{tikzpicture}
\caption{Example~\ref{ex-ERN}: Extreme Reissner-Nordstrom.}
\label{fig-ERN}
\end{figure}

\begin{example}\label{ex-min-apart}
A typical Brill-Lindquist geometrostatic manifold, $M$, satisfying the hypothesis
of Theorems~\ref{minsurf} and ~\ref{Almost-Rigidity} 
is depicted in Figure~\ref{fig-min-apart}.   Here
we have three black holes with small masses.  The $p_i$ are located sufficiently
far apart that they have distinct outermost minimal surfaces, $\Sigma_i$.  The outermost region, $M'$, lies above $\Sigma=\bigcup_i \Sigma_i$ in this image.
If $\beta_i << \alpha_i$
then the necks can be quite long as depicted here.  Even with long necks, Theorem~\ref{Almost-Rigidity} implies that $M'$ is close in the intrinsic flat sense to ${\mathbb{E}}^3$.  The intrinsic flat distance essentially measures a volume
between Euclidean space and $M'$, and we prove these necks have small volume.
\end{example}

\begin{figure}[h]
\centering
\begin{tikzpicture}[scale=.5]
\draw[thick] (-8.5,-6) to (7.5, -6) to (9.5,-4) to (-6.5, -4) to (-8.5, -6);
\draw[thick] (-8.25,-9.25) to (-5.25, -9.25) to (-4.25,-7.25) to (-6.1, -7.25); 
\draw[thick] (-6.35, -7.25) to (-7.25, -7.25) to (-8.25, -9.25);
\draw[thick] (-2.25,-8.55) to (0.75, -8.55) to (1.75,-6.85) to (-0.1, -6.85); 
\draw[thick] (-0.35, -6.85) to (-1.25, -6.85) to (-2.25, -8.55);
\draw[thick] (3.75,-9) to (6.75, -9) to (7.75,-7) to (5.9, -7); 
\draw[thick] (5.65, -7) to (4.75, -7) to (3.75, -9);
\draw[thick, dashed] (-6.6, -5) to [out=-60, in=95] (-6.35, -6);
\draw[thick] (-6.35, -6) to [out=-85, in=90] (-6.35, -6.25) to  [out=-90, in=60](-6.75, -8.25);
\draw[thick, dashed] (-0.6, -5) to [out=-60, in=95] (-0.4, -6);
\draw[thick] (-0.4, -6) to [out=-85, in=90] (-0.4, -6.25) to  [out=-90, in=60](-0.75, -8);
\draw[thick, dashed] (5.4, -5) to [out=-60, in=95] (5.6, -6);
\draw[thick] (5.6, -6) to [out=-85, in=90] (5.6, -6.25) to  [out=-90, in=60](5.25, -8.25);
\draw[thick, dashed] (-5.75, -5) to [out=-120, in=85] (-6.15, -6);
\draw[thick] (-6.15, -6) to [out=-95, in=90] (-6.15, -6.5) to  [out=-90, in=120](-5.85, -8.25);
\draw[thick, dashed] (0.25, -5) to [out=-120, in=85] (-.15, -6);
\draw[thick] (-.15, -6) to [out=-95, in=90] (-.15, -6.5) to  [out=-90, in=120](0.15, -8);
\draw[thick, dashed] (6.25, -5) to [out=-120, in=85] (5.87, -6);
\draw[thick] (5.87, -6) to [out=-95, in=90] (5.85, -6.5) to  [out=-90, in=120](6.15, -8.25);
\draw[ultra thick] (-6.35,-6.65) to [out=-15, in=-165] (-6.15,-6.65);
\node[above] at (-5.35,-7.2) {$\Sigma_1$};
\draw[ultra thick] (-0.35,-6.35) to [out=-15, in=-165] (-0.15,-6.35);
\node[above] at (.65,-7) {$\Sigma_2$};
\draw[ultra thick] (5.65,-6.5) to [out=-15, in=-165] (5.85,-6.5);
\node[above] at (6.55,-7.2) {$\Sigma_3$};
\node[above] at (11,-5.25) {$|x|\to\infty$};
\node[below] at (-6.75,-9.25) {$x\to p_1$};
\node[below] at (-0.75,-8.75) {$x \to p_2$};
\node[below] at (5.25,-9.25) {$x \to p_3$};
\end{tikzpicture}
\caption{Example~\ref{ex-min-apart}: $M$ as in Theorems~\ref{minsurf} and ~\ref{Almost-Rigidity}.  }
\label{fig-min-apart}
\end{figure}
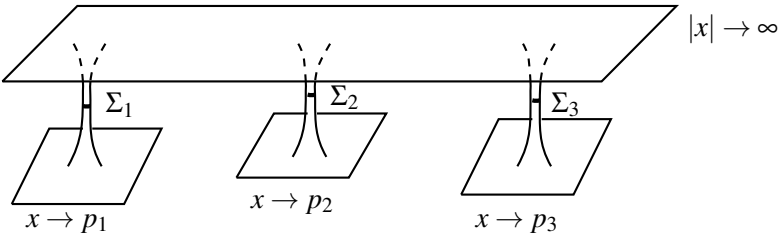

\begin{example}\label{ex-min-share}
Brill and Lindquist demonstrated numerically that if the masses of two
black holes are sufficiently large relative to the distance between them, then
they share a single
outermost minimal surface, $\Sigma_1=\Sigma_2$.  Such a manifold, $M$, which
fails to satisfy the hypothesis of Theorems~\ref{minsurf} 
and ~\ref{Almost-Rigidity}, is depicted in
 Figure~\ref{fig-min-share}.   The outermost region, $M'$, lies above the $\Sigma_i$ in this image.
\end{example}

\begin{figure}[h]
\centering
\begin{tikzpicture}[scale=.5]
\draw[thick] (-5,-5) to (5, -5) to (7,-3) to (-3, -3) to (-5, -5);
\draw[thick, dashed] (-.6, -4.25) to [out=-60, in=120] (-0.25, -5);
\draw[thick] (-.25, -5) to [out=-60, in=90] (0, -6) to  [out=-90, in=60](-1, -8.25);
\draw[thick, dashed] (2.1, -4.25) to [out=-120, in=60] (1.65, -5);
\draw[thick] (1.65, -5) to [out=-120, in=90] (1, -6.5) to  [out=-90, in=120](1.5, -8.25);
\draw[thick] (-.8, -8.5) to [out=60, in=120] (1.3, -8.5);
\draw[thick] (-1, -7) to (0, -8.5) to (-1, -10) to (-2, -8.5) to (-1, -7);
\draw[thick] (1.5, -7) to (2.5, -8.5) to (1.5, -10) to (.5, -8.5) to (1.5, -7);
\draw[ultra thick] (0,-6.25) to [out=-15, in=-165] (1,-6.25);
\node[right] at (1.5,-6.1) {$\Sigma_1=\Sigma_2$};
\node[right] at (7,-3.5) {$|x|\to\infty$};
\node[right] at (-4.5,-10) {$x\to p_1$};
\node[right] at (2,-10) {$x\to p_2$};
\end{tikzpicture}
\caption{Example~\ref{ex-min-share}: Failing the hypothesis of Theorem~\ref{minsurf} and ~\ref{Almost-Rigidity}.}
\label{fig-min-share}
\end{figure}
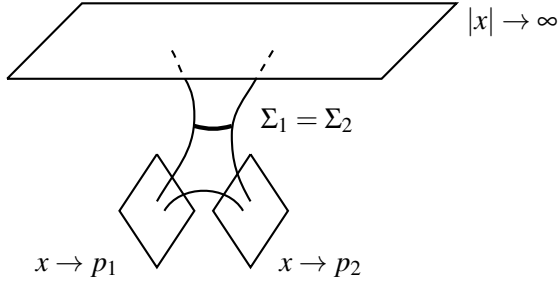

\begin{example}\label{ex-lim-ERN}
In Figure~\ref{fig-lim-ERN} we see a sequence of outermost regions, $M'_k$,
satisfying the hypothesis of Theorem~\ref{Almost-Rigidity}.  In particular,
$m_{ADM}(M'_k) \to 0$.  Here $n=1$, $\alpha_1>0$ and $\alpha_1>>\beta_1>0$
for each $M_k$, and so we have thinner and thinner necks which can be
quite long.  In the limit, the neck shrinks to a line segment which has no
volume at all and thus disappears under intrinsic flat convergence.
\end{example}

\begin{figure}[h] 
\centering
\begin{tikzpicture}[scale=.5]

\draw (-3, -4) to (-3, -8); 
\draw (-3.2, -4.2) to (-3, -4) to (-2.8, -4.2);
\draw (-3.2, -7.8) to (-3, -8) to (-2.8, -7.8);
\node[left] at (-3, -6) {$\approx L$};

\draw[thick] (-2,-5) to (2, -5) to (3,-3.5) to (-1, -3.5) to (-2, -5);
\draw[thick, dashed] (-.65, -4.25) to [out=-75, in=115] (-0.4, -5);
\draw[thick] (-.4, -5) to [out=-75, in=90] (-0.01, -6.5) to (0,-8);
\draw[thick, dashed] (1.65, -4.25) to [out=-115, in=75] (1.3, -5);
\draw[thick] (1.3, -5) to [out=-115, in=90] (0.77, -6.5) to (0.76,-8);
\draw[thick] (0,-8) to [out =-15, in=-165] (0.76, -8);

\draw[thick] (3,-5) to (7, -5) to (8,-3.5) to (4, -3.5) to (3, -5);
\draw[thick, dashed] (4.75, -4.25) to [out=-75, in=115] (5, -5);
\draw[thick] (5, -5) to [out=-75, in=90] (5.29, -6.5) to (5.3,-8);
\draw[thick, dashed] (6.25, -4.25) to [out=-115, in=75] (5.9, -5);
\draw[thick] (5.9, -5) to [out=-115, in=90] (5.57, -6.5) to (5.55,-8);
\draw[thick] (5.3,-8) to [out =-15, in=-165] (5.55, -8);

\draw[thick] (8,-5) to (12, -5) to (13,-3.5) to (9, -3.5) to (8, -5);
\draw[thick, dashed] (10.5, -4.25) to [out=-90, in=90] (10.5, -5);
\draw[thick] (10.5, -5) to [out=-90, in=90] (10.5,-8);

\draw (14, -4) to (14, -8); 
\draw (13.8, -4.2) to (14, -4) to (14.2, -4.2);
\draw (13.8, -7.8) to (14, -8) to (14.2, -7.8);
\node[right] at (14, -6) {$=L$};
\end{tikzpicture}
\caption{Example~\ref{ex-lim-ERN}: A sequence of outermost regions 
as in Theorem~\ref{Almost-Rigidity}}. 
\label{fig-lim-ERN}
\end{figure}
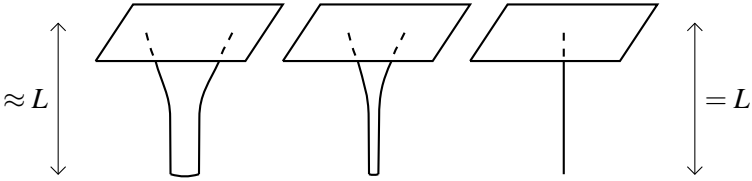

\section{Minimal Surfaces} \label{sect-minsurf}

In this section we locate the outermost minimal surfaces, $\Sigma_i$,
proving Theorem~\ref{minsurf}.   In Subsection~\ref{sect-HI} we 
review Huisken-Ilmanen which proves
the existence of the outermost minimal surfaces.  In Subsection~\ref{curv-inj}
we prove that an appropriately rescaled annular region within a geometrostatic manifold has bounded curvature and injectivity radius. These bounds 
are applied in Subsection~\ref{CM:review} to provide a
lower bound on $\area_g(\Sigma_i)$. In Subsection~\ref{part} we
prove that for all $\Sigma_i$ there is some $p\in P$ such that $\Sigma_i\,\cap \, \bar{B}_\delta(p, \sqrt{\area_g(\Sigma_i)/\pi}) \neq \emptyset$.  In Subsection~\ref{whole} we combine these results and the Penrose Inequality to prove that if the separation factor $\sigma$ is large compared to the ADM mass $m$, then
$ \Sigma_i \subseteq B_\delta(p, 2C_1\sqrt{\area_g(\Sigma_i)/\pi})$ for some $p\in P$.  
In Subsection~\ref{not-close} we apply the inversion proven in the Appendix to
flip $p\in P$ to $\infty$ to prove that $\Sigma_i$ is not too close to any $p\in P$, thus completing the proof of Theorem~\ref{minsurf}.

\subsection{A Review of Huisken-Ilmanen}\label{sect-HI}

in \cite{Huisken-Ilmanen}, Huisken and Ilmanen provide a rigorous definition of an
outermost minimal surface and prove its regularity.   We
review this here.

Let $M$ be a complete $3$-manifold with asymptotically flat ends.
Let $K_1$ be the closure of the union of the images of all smooth, compact, immersed minimal surfaces in M.   They observe that since the region near infinity is foliated by spheres of positive mean curvature, $K_1$ is compact.   The trapped region $K$ is defined to be the union of $K_1$ together with the bounded components of $M\setminus K_1$. The set $K$ is clearly compact as well. 

Let $M'$ be any connected component of $M\setminus K$.  This is considered to be an ``exterior region".  It has one asymptotically flat end and a compact boundary.   In our paper we are considering specifically $M'$ corresponding to the end as $r\to \infty$.   

In Lemma 4.1 of \cite{Huisken-Ilmanen}, Huisken and Ilmanen
prove that $M'$ is connected and asymptotically flat, has a compact, minimal boundary, and contains no other compact minimal surfaces (even immersed).   In addition $M'$ is diffeomorphic to $\mathbb{R}^3$ minus a finite number of regions diffeomorphic to open 3-balls with disjoint closures. The boundary of $M'$ minimizes area in its homology class.  This $M'$ is the outermost region we have defined in our introduction.

Let $\Sigma$ be any connected component of $\partial M'$.  Huisken-Ilmanen proved that
\be \label{Penrose-HI}
m_{ADM}(M',g) \ge \sqrt{ \frac{\area_g(\Sigma)}{16\pi} \,}
\ee  
which implies the Penrose Inequality if $\partial M'$ were connected \cite{Huisken-Ilmanen}.
Bray \cite{Bray-Penrose} proved the Penrose Inequality as in (\ref{Penrose}) 
even when the boundary
has more than one connected component.

Applying this to our paper, we have an outermost or exterior region
\be
M'= {\mathbb{R}}^3\setminus \bigcup_{\alpha} U_\alpha,
\ee
 where $U_\alpha$ are diffeomorphic to three dimensional balls with stable minimal boundaries, $\Sigma_\alpha=\partial U_\alpha$.  Note that every $p_i$
 must lie in one of the $U_\alpha$ and that if $p_i \in U_\alpha$ then
 $U_\alpha=\Omega_\alpha$ of Definition~\ref{outermost-surface}.  Recall that it is possible that some $\Omega_i=\Omega_j$. It is possible that there are some additional $U_\alpha$ which do not contain any $p_i$ for $i=1$ to $n$.  We set these $U_\alpha=\Omega_i$ with $i\le 0$ so that we may simply write:
\be \label{M'defined}
M' =  {\mathbb{R}}^3\setminus \bigcup_{i=-n'}^n \Omega_i
\textrm{ and } \partial M' =  \bigcup_{i=-n'}^n \Sigma_i
\textrm { where } \Sigma_i = \partial \Omega_i.
\ee
Observe that we have
\be\label{disjointOmega}
\Omega_i\neq\Omega_j \Longrightarrow \Omega_i\,\cap \,\Omega_j=\emptyset.
\ee 

\begin{conj} We conjecture that for every $\Omega_\alpha$ there is a $p_i\in \Omega_\alpha$ so that $M' =  {\mathbb{R}}^3\setminus \bigcup_{i=1}^n \Omega_i$.
\end{conj}

\subsection{Minimal surfaces in a Conformally Flat Manifold}

Suppose
\be\label{conf-flat}
\bigg(\,M\,,\,g_{\textcolor{white}{0}}\bigg)=\bigg(\,\,{\mathbb{R}^3}\setminus\,
P \,,\, \Psi^2\, \delta\,\bigg)
\ee
where $P=\{p_1,...p_n\}$,  $\delta$ is the Euclidean metric, and 
$\Psi>0$ on ${\mathbb{R}^3}\setminus P$.   We will not require
 $\Psi^2=(\,\chi\,\psi)^2$ in the beginning of this subsection.

If $\Sigma$ is a closed surface in this manifold the area of $\Sigma$ with
respect to the metric $g$ is
\be
\area_g(\Sigma)=\int_{x\in \Sigma}\, d\sigma_g=\int_{x\in \Sigma} \Psi^2(x) \,d\sigma_\delta.
\ee

If we vary $\Sigma_t$ with respect to an arbitrary variational field ${\bf v}=f {\bf{n}}$ where $\bf{n}$ is the outward normal, we see that
\be
\frac{d}{dt} \area_g(\Sigma_t)|_{t=0}=\int_{x\in \Sigma} 
\left(2\Psi(x) \nabla \Psi_x\cdot {\bf n} + \Psi^2(x) H_x\right) f(x) \, d\sigma_\delta
\ee
where $H_x$ is the outward pointing mean curvature of $\Sigma$ as a submanifold of $(M, \delta)$ and $\nabla=\nabla_\delta$ is the gradient with respect to the Euclidean metric $\delta$.  Thus if $\Sigma_t$ is a minimal surface in $(M, g)$ then
\be
2\Psi(x) \nabla \Psi_x\cdot {\bf n} + \Psi^2(x) H_x=0 \qquad \forall x \in \Sigma
\ee
and
\be \label{mean-curv-1}
H_x= - 2 \frac{\nabla \Psi_x\cdot {\bf n}}{\Psi(x)} \qquad \forall x \in \Sigma.
\ee

On any surface, the mean curvature is the sum of the principal curvatures,
$H= \lambda_1+\lambda_2$, and while the Gauss curvature $K=\lambda_1\lambda_2$ and so $H^2-4K \ge 0$.   By Gauss-Bonnet, the
Euler charateristic satisfies
\be\label{GB}
2\pi\chi(\Sigma)=  \int_\Sigma K d\sigma_\delta \le \frac{1}{4}\int_\Sigma H^2 d\sigma_\delta.
\ee 
Combining this with (\ref{mean-curv-1}) we see that any minimal surface satisfies:
\be
2\pi \chi(\Sigma) \le   \frac{1}{4}\int_\Sigma \left(- 2 \frac{\nabla \Psi_x\cdot {\bf n}}{\Psi(x)}\right)^2 d\sigma_\delta
\ee
which also implies that
\begin{eqnarray}
2\pi \chi(\Sigma)
&\le&  \int_\Sigma \frac{(\nabla \Psi_x\cdot {\bf n})^2}{\Psi^2(x)} d\sigma_\delta\\
&\le&  \int_\Sigma \frac{|\nabla \Psi_x|^2}{\Psi^4(x)} \Psi^2(x) d\sigma_\delta\\
&\le&  \,\max_{x\in \Sigma}\left\{ \frac{|\nabla \Psi_x|^2}{\Psi^4(x)}\right\} \int_\Sigma \Psi^2(x) d\sigma_\delta\\
&=&  \,\max_{x\in \Sigma}\left\{ \frac{|\nabla \Psi_x|^2}{\Psi^4(x)}\right\} \area_g(\Sigma).
\end{eqnarray}
Note that we have equality iff $\frac{|\nabla \Psi_x|^2}{\Psi^4(x)}$ is constant on
$\Sigma$, $(\nabla \Psi\cdot {\bf n})=|\nabla \Psi|$ and $H^2=4K$
where $H$ and $K$ are the mean and Gauss curvatures of the surface 
with respect to the Euclidean metric on $M$.  Note further that $(\nabla \Psi\cdot {\bf n})=|\nabla \Psi|$ iff $\nabla \Psi$ is perpendicular to the surface iff
$\Psi$ is constant on the surface.  This implies the following proposition:

\begin{prop}\label{prop-area-below}
The area of a minimal surface in $(M,g)$ as in
(\ref{conf-flat}) is bounded below by:
\be\label{area-below}
\area_g(\Sigma)\ge 2\pi \chi(\Sigma)\left(\,\max_{x\in \Sigma}\left\{ \frac{|\nabla \Psi_x|^2}{\Psi^4(x)}\right\}\right)^{-1}
\ee
and we have equality iff $\Psi$ and $|\nabla \Psi_x|$ are constant on
$\Sigma$ and $H^2=4K$
where $H$ and $K$ are the mean and Gauss curvatures of the surface 
with respect to the Euclidean metric on $M$.
\end{prop}

In particular when $g$ is a metric with positive scalar curvature, we know by Huisken-Ilmanen that any
outward minimizing surface is smooth minimal surface diffeomorphic to a sphere.  
So we have (\ref{area-below}) with $\chi(\Sigma)=2$.  Furthermore, by the Penrose Inequality, as proven
in Huisken-Ilmanen, we have
\be
\area_g(\Sigma)\le 16\pi m^2 \textrm{ where } m=m_{ADM}(M)
\ee
with equality iff $M$ is isometric to Schwarzschild space.  Combining this with the previous proposition we have the following:

\begin{prop}\label{prop-mass-below}
If $(M,g)$ as in (\ref{conf-flat}) has positive scalar curvature and $\Sigma$
is outward minimizing then
\be\label{mass-below}
16\pi m^2 \ge 4\pi \left(\,\max_{x\in \Sigma}\left\{ \frac{|\nabla \Psi_x|^2}{\Psi^4(x)}\right\}\right)^{-1}
\ee
and we have equality iff $\Psi$ and $|\nabla \Psi_x|$ are constant on
$\Sigma$ and $H^2=4K$ and $(M,g)$ is isometric to Schwarzschild space.
\end{prop}

This immediately implies the following:

\begin{prop}\label{lower-mass-geo}
Let $(M,g)$ be a geometrostatic manifold with $\Psi(x)=\psi(x)\chi(x)$
satisfying (\ref{chi-psi}).   Then by (\ref{mean-curv-1}) any minimal surface
$\Sigma$ satisfies:
\be \label{mean-curv}
H_x= - 2 \left(\frac{\nabla \psi_x\cdot {\bf n}}{\psi(x)}+ \frac{\nabla \chi_x\cdot {\bf n}}{\chi(x)}\right) \qquad \forall x \in \Sigma.
\ee
If in addition the minimal surface is outward minimizing, we have
\begin{eqnarray}
4\pi &\le&  \int_\Sigma \left(\frac{\nabla \psi_x\cdot {\bf n}}{\psi(x)}+ \frac{\nabla \chi_x\cdot {\bf n}}{\chi(x)}\right)^2 d\sigma_\delta\\
&\le& \area_g(\Sigma) \,\max_{x\in \Sigma} 
\left(\frac{\nabla \psi_x\cdot {\bf n}}{\psi^2(x)\chi(x)}+ \frac{\nabla \chi_x\cdot {\bf n}}{\chi^2(x)\psi^(x)}\right)^2 \\
&\le& 16 \pi m^2 \,\max_{x\in \Sigma} 
\left(\frac{|\nabla \psi_x|}{\psi^2(x)\chi(x)}+ \frac{|\nabla \chi_x|}{\chi^2(x)\psi(x)}\right)^2 
\end{eqnarray}
and we have equality iff $\psi(x)\chi(x)$ and $\nabla(\psi(x)\chi(x))$ are constant on
$\Sigma$  and $H^2=4K$
and $(M,g)$ is isometric to Schwarzschild space.
\end{prop}

\begin{example}
We can apply this proposition to Schwarzschild space $M_{\mathrm{Sch}}$, satisfying (\ref{Sch}) with mass $m=\alpha_1+\beta_1=2\alpha_1$ to
verify that the level set $\rho_1^{-1}(m/2)$ is a minimal surface and prove it
is the only outermost minimizing surface in $M_{\mathrm{Sch}}$.
On  $\rho_1^{-1}(m/2)$ we have 
\be
\chi(x)=\psi(x)=1+\beta_1/\rho_1(x) = 1+ (m/2)/(m/2) = 1+ 1=2
\ee
 and
\be
\nabla \chi_x=\nabla\psi_x=( -\beta_1/\rho_1^2(x))\nabla \rho_1=
-(m/2)/(m/2)^2\nabla \rho_1= -(2/m)\nabla \rho_1
\ee
Since $H_x= 2/\rho_1=2/(m/2)=4/m$ and 
\be 
- 2 \left(\frac{\nabla \psi_x\cdot {\bf n}}{\psi(x)}+ \frac{\nabla \chi_x\cdot {\bf n}}{\chi(x)}\right)= - 2 (-2/m)(1/2) - 2 (-2/m)(1/2)= 4/m,
\ee
we have (\ref{mean-curv}) and $\rho_1^{-1}(m/2)$ is a minimal surface in $(M,g)$.
So now we know our outermost region $M' \subset \rho_1^{-1}[m/2, \infty)$.

Next suppose $\Sigma \subset M'$ were an outermost
minimizing surface, then by Proposition~\ref{lower-mass-geo} we have
\begin{eqnarray}
4\pi &\le& 16 \pi m^2 \,\max_{x\in \Sigma} 
\left(\frac{|\nabla \psi_x|}{\psi^2(x)\chi(x)}+ \frac{|\nabla \chi_x|}{\chi^2(x)\psi(x)}\right)^2 \\
&= & 16 \pi m^2 \,\max_{x\in \Sigma} \left(\frac{2|\nabla \psi_x|}{\psi^3(x)}\right)^2\\
& = & 16 \pi m^2 \,\max_{x\in \Sigma} 
\frac{4\,| -(m/2)/\rho_1^2(x)\,|^2\,|\nabla \rho_1 |^2}{|1+(m/2)/\rho_1(x)|^6}\\
&=& 16 \pi m^2 \,\max_{x\in \Sigma}\frac{(m \rho_1(x))^2}{|\rho_1(x)+(m/2)|^6}
\qquad \textrm{ because } |\nabla \rho_1 |=1\\
& = & 16 \pi m^2 \,\max_{x\in \Sigma} F^2(\rho_1(x), m/2)
\textrm{ where  } F(\rho,\beta)= 2 \rho\beta/(\rho+\beta)^3.
\end{eqnarray}
For fixed $\beta=\beta_1=m/2$, $F(\rho, \beta)$ converges to $0$ as 
$\rho\to 0$ and as $\rho\to \infty$.  F increases to a single critical point at 
$\rho=\beta/2=m/4$ and then decreases.
Since $\Sigma\subset M'$ implies $\rho_1(x)\ge \beta_1=m/2$, the maximum of
$F$ occurs at $\rho=m/2$.  Thus we have
\begin{eqnarray}
4\pi &\le& 16 \pi m^2 \,\max_{\rho \in [m/2,\infty)} F^2(\rho, m/2)
\le 16 \pi m^2 F^2(m/2, m/2)\\
&=& 16 \pi m^2 \left(\frac{2(m/2)(m/2)}{((m/2)+(m/2))^3}\right)^2 
= 16 \pi m^2  \left(\frac{m^2/2}{m^3}\right)^2   
= 4\pi.
\end{eqnarray}
Thus we have equality in Proposition~\ref{lower-mass-geo}, which implies that
\be
\frac{m^2 \rho_1^2(x)}{|\rho_1(x)+(m/2)|^6}=\frac{m^2 (m/2)^2}{|(m/2)+(m/2)|^6}
\textrm{ for all } x\in \Sigma.
\ee
So $\rho_1(x)=m/2$ for all $x \in \Sigma$.  Thus we have confirmed that 
$\rho_1^{-1}(m/2)$ is the only outermost minimizing surface in $M_{\mathrm{Sch}}$.
\end{example}

We conjecture more generally that if $(M,g)$ is a geometrostatic manifold as in 
(\ref{M}), then the only outermost minimizing surfaces are the $\Sigma_i$ defined in
$\ref{outermost-surface}$.

\subsection{Estimates on the curvature and the injectivity radius}
\label{curv-inj}

In this section we prove that an appropriately rescaled annular region within a geometrostatic manifold has bounded curvature and injectivity radius. These bounds will be applied later to locate the outermost minimal surfaces in these manifolds.

Fix a geometrostatic manifold $(\mathbb{R}^3\setminus P, g)$ and fix some $i$ with $1\le i\le n$.   We assume 
there is some positive length, $c>0$, such that
\be\label{choice-c-ann}
\sigma_i = \min_{j\neq i} \{ |p_i-p_j|\} > 5c
\quad\textrm{ or } \quad\sigma_i^{out}=\,\max_{j\neq i}\{|p_i-p_j|\} < \tfrac{1}{5}c.
\ee
Consider a Euclidean
annulus 
\be\label{annulus}
\mathcal{A}=\{u\in \mathbb{R}^3\ \big|\ \tfrac{1}{4}\le |u|\le 4\} 
\ee
and the mapping 
\be\label{scaledpullback:defn}
\Phi: \mathcal{A}\to \mathbb{R}^3\setminus P \text{\ \ by\ \ } u\mapsto p_i+c u.
\ee
By our choice of $c$ in (\ref{choice-c-ann}) we know that
$\Phi(\mathcal{A})$ avoids $P$ by a definite amount.  In fact
$\Phi^{-1}(P\setminus\{p_i\})$ either lies beyond the outer ring of the annulus
when $\sigma_i < 5c$
or lies within the inner ring when $\sigma^{out}_i> \frac{1}{5} c$.

The scaled pullback metric $g_{c,\Phi}:=c^{-2}\Phi^*g$ 
on the annulus is easily seen to be 
\be\label{pullbackformula}
\left(1+\frac{\alpha_i/c}{|u|}+\varphi_\alpha(u)\right)^2
\left(1+\frac{\beta_i/c}{|u|}+\varphi_\beta(u)\right)^2\delta,
\ee
where
\be
\varphi_\alpha(u)=\sum_{j\neq i} \frac{\alpha_j}{|cu+p_i-p_j|} \text{\ \ and\ \ }
\varphi_\beta(u)=\sum_{j\neq i} \frac{\beta_j}{|cu+p_i-p_j|}.
\ee

Before we state the main results of this section, we introduce some terminology which will make our statements more efficient.  

\begin{defn}\label{psi-u} {\em
Let $\psi(u)$ be a function (or a tensor field) defined on the annulus $\mathcal{A}=\{u\in \mathbb{R}^3\ \big|\ \tfrac{1}{4}\le |u|\le 4\}$. We say that \emph{$\psi$ is controllable by $K_0$} whenever there exists a positive increasing function $P$, independent of $\psi$, such that 
\be
\|\psi\|_{C^0(\mathcal{A})}\le P(K_0).
\ee
Furthermore, we say that \emph{$\psi$ is controllable by $K_0$ with all of its derivatives} whenever all of the Euclidean partial derivatives $\partial^l_u$ (of the components of $\psi(u)$) are controllable by $K_0$. }
\end{defn}

Here are the two main results of this subsection. 

\begin{prop}\label{curvlemma:oct}
Assume there is a positive length, $c>0$, that satisfies (\ref{choice-c-ann})
and consider the scaled pullback metric $g_{c,\Phi}$ on
$\mathcal{A}$
defined in \eqref{scaledpullback:defn}--\eqref{pullbackformula}. 
Observe that
\begin{enumerate}
\item $g_{c,\Phi}\ge \delta$. 
\medbreak
\item There exist constants $k_1$ and $k_2$ such that $k_1\delta \le g_{c,\Phi}\le k_2\delta$ and such that both $\frac{k_1}{k_2}$ and $\frac{k_2}{k_1}$ are controllable by $m/(\sigma_i+c)$. 
\medbreak
\item The Christoffel symbols of $g_{c,\Phi}$ are controllable by 
$m/(\sigma_i+c)$ with all of their derivatives. 
\medbreak
\item The (sectional) curvature of $g_{c,\Phi}$ is controllable by 
$m/(\sigma_i+c)$ with all of its derivatives. 
\end{enumerate}
\end{prop}

We also control the injectivity radius.  Note that Cheeger-Gromov-Taylor
estimate the injectivity radius in a far more general setting in \cite{ChGrTa}.

\begin{prop}
\label{injrad:march}
Assume there is a positive length, $c>0$, that satisfies (\ref{choice-c-ann})
and let $m/(\sigma_i+c)\le 1$. Consider the scaled pullback metric $g_{c,\Phi}$ defined in \eqref{scaledpullback:defn}--\eqref{pullbackformula}. There is a uniform lower bound, $i_0>0$, on the injectivity radii of $g_{c,\Phi}$ over $\mathcal{A}'=\{u\in \mathbb{R}^3\big| \tfrac{1}{2}\le |u|\le 2\}$. 
\end{prop}

\begin{rmrk}
Note that the explicit bound, $m/(\sigma_i+c)\le 1$, given in Proposition~\ref{injrad:march} is not optimal, and has only been chosen for simplicity.   
\end{rmrk}

Before proving these propositions, we first prove a series of general lemmas.

\begin{lemma}\label{partial-r}
Let $\nu\in \mathbb{R}$. The $l$-th order partial derivatives of $\xi\mapsto |\xi|^\nu$ on $\mathbb{R}^3$ satisfy point-wise estimate
\be
|\partial^l (|\xi|^\nu)|\le C_{l,\nu} |\xi|^{\nu -l},
\ee
where the constant $C_{l,\nu}$ depends only on $l$ and $\nu$.
\end{lemma}

\begin{proof}
We first prove the lemma in the case of $\nu=1$:
\be\label{basecase}
|\partial^l (|\xi|)|\le C(l)\cdot |\xi|^{1-l},
\ee
where the constant $C(l)$ depends only on $l$.
We do so by induction on $l$. A direct computation verifies the base cases $l\le 2$. For $l\ge 3$ the claim follows from the inductive hypothesis and 
\be
\partial^l(|\xi|^2)=2|\xi| \partial^l(|\xi|) +\sum_{i=1}^{i=l-1}{{l}\choose{i}}\partial^i(|\xi|)\partial^{l-i}(|\xi|). 
\ee
For general values of $\nu$ observe that the derivative $\partial^l(|\xi|^\nu)$ is a linear combination of terms of the form 
\be
|\xi|^{\nu-k}\partial^{l_1}(|\xi|)...\partial^{l_k}(|\xi|),\ \ 1\le k\le l
\ee
with coefficients which depend only on $\nu$ and positive integers $l_1$, ..., $l_k$ which satisfy $l_1+...+l_k=l$. The claim of our lemma is now a consequence of \eqref{basecase}.
\end{proof}

\begin{lemma}\label{derivbds:lemma}
Let $\varphi\ge 0$ be a smooth function defined on the annulus,
$\mathcal{A}$, defined in (\ref{annulus})
Let $a$ be a positive real number and let 
\be
f(u)=\ln(1+\tfrac{a}{|u|}+\varphi(u)).
\ee
For each integer value of $l\ge 0$ there exist polynomials $P_l$ whose (positive) coefficients are independent of $a$ and $\varphi$ such that 
\be
\|df\|_{C^l(\mathcal{A})}\le P_l(\|\varphi\|_{C^{l+1}(\mathcal{A})}).
\ee
\end{lemma}

\begin{proof}
Consider the function $\tilde{f}=e^f$. A straightforward induction argument shows that the components of the $l$-th derivatives of $df$ are polynomials in $\tilde{f}^{-1}\partial^i \tilde{f}$ whose coefficients depend only on $l$. Thus, it suffices to prove point-wise bounds on $\tilde{f}^{-1}\partial^i \tilde{f}$ in terms of 
$\|\varphi\|_{C^{i}(\mathcal{A})}$ and constants which depend only on $i$.
 
By virtue of the fact that the annulus $\mathcal{A}$ is compact and bounded away from the origin we know that there is a constant $c_i$ depending only on $i$ such that 
\be
|\partial^i\tilde{f}|\le 1+c_ia+\|\varphi\|_{C^{i}(\mathcal{A})}.
\ee
Thus, 
\be
\left|\frac{\partial^i \tilde{f}}{\tilde{f}}\right|\le \frac{1+c_i a}{1+a/4}+\|\varphi\|_{C^{i}(\mathcal{A})}\le 4(1+c_i)+\|\varphi\|_{C^{i}(\mathcal{A})}.
\ee
This completes our proof. 
\end{proof}

\begin{lemma}\label{varphi:ests}
Assume there is a positive length, $c>0$, that satisfies (\ref{choice-c-ann}) and consider the scaled pullback metric $g_{c,\Phi}$ 
on $\mathcal{A}$
defined in \eqref{scaledpullback:defn}--\eqref{pullbackformula}.
Then for every integer $l\ge 0$ there exists a constant $C_l$ which depends only on $l$ such that 
\be
\|\varphi_\alpha\|_{C^{l}(\mathcal{A})}, \|\varphi_\beta\|_{C^{l}(\mathcal{A})}
\le C_l \tfrac{m}{\sigma_i + c}.
\ee
\end{lemma}

\begin{proof}
As a consequence of Lemma \ref{partial-r} we have that  
\be \label{varphi:ests:1}
\left|\partial_u^l \left(\frac{\alpha_j}{|cu+p_i-p_j|}\right)\right|\le C'_l\frac{\alpha_j c^l}{|cu+p_i-p_j|^{l+1}},
\ee
where the constant $C'_l$ depends only on $l$. 

There are two cases in our hypothesis that $c$ satisfies (\ref{choice-c-ann}).
In the case where $\sigma_i  > 5c$, we have $|p_i-p_j|>5c$ for all $j\neq i$,
and so we have 
\be
|cu+p_i-p_j|\ge |p_i-p_j|-c|u|\ge \sigma_i - 4c>
\sigma_i - \tfrac{4\sigma_i}{5}=\tfrac{\sigma_i}{5}> \tfrac{\sigma_i+c}{10}
> \tfrac{\sigma_i+c}{40}.
\ee
In the case where $\sigma_i^{out} < \tfrac{1}{5}c$ we have
$|p_i-p_j|< c/5$ and so we have
\be
|cu+p_i-p_j|\ge c|u|-|p_i-p_j| \ge \tfrac{c}{4} - \tfrac{c}{5}=\tfrac{c}{20}
> \tfrac{\sigma_i+c}{40}
\ee
because $\sigma_i<\sigma_i^{out}$.

In both cases, it follows from (\ref{varphi:ests:1}) and $\sum_j \alpha_j<m$ 
that 
\be\label{rndbound}
\left|\partial^l_u \varphi_\alpha(u)\right|
\le C'_l \sum_{j\neq i} \alpha_j\frac{c^l}{\,\,\,\,\left(\frac{c+\sigma_i}{40}\right)^{l+1}}
< C'_l  \left(\tfrac{m}{\sigma_i+c}\right)\left(\tfrac{c}{\sigma_i+c}\right)^l
< C'_l \left(\tfrac{m}{\sigma_i+c}\right).
\ee
The same argument applies to $\varphi_\beta$. 
\end{proof}

We now prove Proposition~\ref{curvlemma:oct}:

\begin{proof}[Proof of Proposition~\ref{curvlemma:oct}]

The first claim that $g_{c,\Phi}\ge \delta$ is immediate from \eqref{pullbackformula}.
 
By (\ref{pullbackformula}), the fact that $\varphi_\alpha>0$ and $\varphi_\beta>0$, the fact that $4>|u|>1/4$ on $\mathcal{A}$,
and Lemma \ref{varphi:ests} with $l=0$ we have
\be
\left(1+\tfrac{\alpha_i}{4c}\right)^2
\left(1+\tfrac{\beta_i}{4c}\right)^2\delta \le g_{c,\Phi} \le  \left(1+4\tfrac{\alpha_i}{c}+C_0\tfrac{m}{\sigma_i+c}\right)^2
\left(1+4\tfrac{\beta_i}{c}+C_0\tfrac{m}{\sigma_i+c}\right)^2\delta.
\ee
Thus we have the second claim of 
Proposition~\ref{curvlemma:oct}. 

To prove the remaining claims we express 
$g_{c,\Phi}$ in the form of $e^{2f}\delta$ where 
\be
f(u)=\ln\left(1+\frac{\alpha_i/c}{|u|}+\varphi_\alpha(u)\right)+\ln\left(1+\frac{\beta_i/c}{|u|}+\varphi_\beta(u)\right)
\ee
for functions $\varphi_\alpha$ and $\varphi_\beta$ of Lemma \ref{varphi:ests}. In fact, by applying Lemmas \ref{derivbds:lemma} and \ref{varphi:ests} we see that there exist polynomials $P_l$ whose (positive) coefficients depend only on $l$ such that 
\be
\|df\|_{C^l(\mathcal{A})}\le P_l\left(\tfrac{m}{\sigma_i+c}\right).
\ee
The claim about the Christoffel symbols of $g_{c,\Phi}$ is now immediate from the fact that $\Gamma_{ij}^k$ can be expressed in terms of components of $df$. 
This also means that the components $R_{ijk}^l$ of the Riemann curvature tensor are bounded by a (universal) polynomial in $\frac{m}{\sigma_i+c}$. 
Since $g_{c,\Phi}^{-1}\le \delta^{-1}$, 
the same applies to the sectional curvatures of $g_{c,\Phi}$.
\end{proof}

We now prove Proposition~\ref{injrad:march}:

\begin{proof}[Proof of Proposition~\ref{injrad:march}]
It follows from Proposition \ref{curvlemma:oct} 
that the Christoffel symbols of $g_{c,\Phi}$ are bounded over $\mathcal{A}$, together with all of their derivatives. Thus, the Cauchy-Picard Theorem implies the uniform time of existence $T$ for all geodesics $\gamma$ 
of $g_{c,\Phi}$ with 
\be
\gamma(0)\in \mathcal{A}',\ \ \ \|\gamma'(0)\|_\delta \le 1.
\ee
In particular, we know that for each $Q\in \mathcal{A}'$ the 
mapping $\exp_Q$ is defined on
\be
B_\delta(0,T)\supseteq B_{g_{c,\Phi}}(0,T).
\ee
The fact that for each such $Q$ the 
mapping $\exp_Q$ is a local diffeomorphism follows from the 
Inverse Function Theorem. In fact, the proof of the Inverse 
Function Theorem shows that if 
\begin{equation}\label{uniformity-inv-func-thm}
\|\mathrm{Id} - D|_v(\exp_Q)\|_\delta <\tfrac{1}{2}
\end{equation}
for all $v\in B_\delta (Q,2r_*)$ then $\exp_Q$ maps diffeomorphically 
onto the ball $B_\delta (Q, r_*)$. 

We proceed by showing that a radius $r_0$ can be 
chosen independently of $Q$ so that the estimate \eqref{uniformity-inv-func-thm} holds for all $v\in B_{g_{c,\Phi}} (Q,2r_0)$. 
Since 
\be
B_\delta(Q, 2r_0/k_2)\subseteq B_{g_{c,\Phi}}(Q,2r_0) \text{\ \ and\ \ } B_{g_{c,\Phi}}(Q,k_1r_0/k_2)\subseteq B_\delta (Q,r_0/k_2)
\ee
for the constants $k_1$ and $k_2$ addressed in Proposition \ref{curvlemma:oct}, the estimate \eqref{uniformity-inv-func-thm}  and
the proof of the Implicit Function Theorem imply that $\exp_Q$ is a diffeomorphism on $B_{g_{c,\Phi}}(Q,k_1r_0/k_2)$. This observation makes the claim of our proposition a consequence of the fact that \eqref{uniformity-inv-func-thm} holds for all $v\in B_{g_{c,\Phi}} (Q,2r_0)$.
\bigbreak
As $D|_v(\exp_Q)$ is identity on the span of $v$, it suffices to 
study 
$D|_v(\exp_Q)$ on the orthogonal complement of $v$. 
There the mapping is given by the Jacobi vector fields 
$Y$ along the $g_{c,\Phi}$-unit speed geodesic 
$\gamma(t)=\exp_Q(t\cdot \tfrac{1}{\|v\|}v)$:
\be
D|_v(\exp_Q): w\mapsto \tfrac{1}{\|v\|} Y(\|v\|),\ \ Y(0)=0,\ \ Y'(0)=w.
\ee
Note that here $\|v\|$ is taken with respect to $g_{c,\Phi}$. 
For the purposes of addressing \eqref{uniformity-inv-func-thm} 
it suffices to work with $w$ which are unit with respect to $\delta$. 
Note that we then have 
\be
k_1\le \|w\|_{g_{c,\Phi}}\le k_2.
\ee

Let $W$ be the $g_{c,\Phi}-$parallel transport of $w$ along $\gamma$ and 
let $\{E_1,E_2, E_3\}$ be any $g_{c,\Phi}-$parallel orthonormal 
frame along $\gamma$. Since the Christoffel symbols of 
$g_{c,\Phi}$ are controlled by $m/(\sigma_i+c) \le 1$ 
with all their derivatives 
and since $\|w\|_\delta=1$, there is a time of existence $T'$ 
(which is independent of our choices of $v$ and $w$) such 
that on $[0,T']$ we have
\be
\|w-W\|_\delta \le \tfrac{1}{4}.
\ee

For $1\le i\le 3$ define the auxiliary functions 
\be
F_i(t)=g_{c,\Phi}(Y-tW, E_i).
\ee
Note that $F_i(0)=F_i'(0)=0$. The Jacobi equation implies 
\begin{eqnarray}
F_i''(t)&=&-R_{g_{c,\Phi}}(Y,\gamma',\gamma',E_i)\\
&=&
-R_{g_{c,\Phi}}(Y-tW,\gamma',\gamma',E_i)-tR_{g_{c,\Phi}}(W,\gamma',\gamma',E_i).
\end{eqnarray}
Temporarily fix a value of $0<t_0<T$. By Proposition~\ref{curvlemma:oct}
the sectional 
curvatures of $g_{c,\Phi}$ are controllable by $m/(\sigma_i+c)<1$, 
and so the same applies to the Jacobi operators 
$R_{g_{c,\Phi}}(., \gamma')\gamma'$. In particular, 
the fact that $\|E_i\|_{g_{c,\Phi}}=1$ and $\|W\|_{g_{c,\Phi}}\le k_2$ 
along $\gamma(t)$, implies  
\be
|F_i''(t)|\le \kappa \left(\sup_{t\le t_0} \|Y-tW\|_{g_{c,\Phi}}+k_2 t\right),\ \ 0\le t\le t_0
\ee
where $\kappa$ denotes the bound on the norms of the 
Jacobi operators $R_{g_{c,\Phi}}(., \gamma', \gamma',E_i)$. 
Upon integration we obtain 
\be
|F_i(t)|\le \kappa \left(t^2 (\sup_{t\le t_0} \|Y-tW\|_{g_{c,\Phi}})+\tfrac{1}{3}k_2 t^3\right).
\ee
As $\|Y-tW\|_{g_{c,\Phi}}\le \sum_{i=1}^3 |F_i(t)|$ we obtain 
\be
\|Y-tW\|_{g_{c,\Phi}}\le 3\kappa (t^2 (\sup_{t\le t_0} \|Y-tW\|_{g_{c,\Phi}})+\tfrac{1}{3}k_2 t^3).
\ee
Under the assumption of $1-3\kappa t_0^2\ge \tfrac{1}{2}$, 
i.e $t_0\le \tfrac{1}{\sqrt{6\kappa}}$, 
and after taking the supremum over $0\le t\le t_0$, 
we arrive at 
\be
\sup_{t\le t_0}\|Y-tW\|_{g_{c,\Phi}}\le 2\kappa k_2 t_0^3.
\ee
Since 
\be
\|Y-t_0W\|_{\delta}\le \sup_{t\le t_0}\|Y-tW\|_{\delta}\le \tfrac{1}{k_1}\sup_{t\le t_0}\|Y-tW\|_{g_{c,\Phi}}
\ee 
and since $t_0<T$ was arbitrary we see that 
\be
\|Y-tW\|_{\delta}\le 2\kappa \tfrac{k_2}{k_1} t^3 \text{\ \ i.e\ \ }
\left\|\tfrac{1}{t}Y(t) -W\right\|_\delta < 2\kappa \tfrac{k_2}{k_1}t^2
\ee
for all $0\le t<\min\{\frac{1}{\sqrt{6\kappa}}, T\}$. 
If in addition $t<T'$, we have 
\be
\|w-\tfrac{1}{t}Y(t)\|_\delta \le 2\kappa \tfrac{k_2}{k_1}t^2+\tfrac{1}{4}.
\ee
It follows that so long as $\|v\|_{g_{c,\Phi}}$ is such that 
\be
0\le \|v\|_{g_{c,\Phi}}<\min\{\tfrac{1}{\sqrt{6\kappa}}, T, T'\}
\ee
and 
\be
2\kappa \tfrac{k_2}{k_1} \|v\|_{g_{c,\Phi}}^2<\tfrac{1}{4}
\ee
 the estimate \eqref{uniformity-inv-func-thm} is fulfilled. The  boundedness of $\tfrac{k_2}{k_1}$ (see Proposition \ref{curvlemma:oct}) implies that the estimate \eqref{uniformity-inv-func-thm} holds for all $v\in B_{g_{c,\Phi}} (Q,2r_0)$ where $r_0$ can be chosen independently of $Q$. This completes our proof. 
\end{proof}

\subsection{The Area of the Minimal Surface} \label{CM:review}

Here we use the estimates in the prior subsection
combined with Colding-Minicozzi's monotonicity formula for the area of a minimal surface to
prove the following theorem depicted in Figure~\ref{fig-S-new}:

\begin{thm}\label{rmrk-s_0}
Fix a length $c>0$.  Let $(M,g)$ be a geometrostatic manifold 
such that for all $i=1,2,...,n$ we have
\be 
\frac{m}{\sigma_i+c}\le 1, 
\ee
and
\be
\qquad \sigma_i> 5c \qquad \textrm{ or }\qquad \sigma_i^{out} < \frac{c}{5}.
\ee
Consider the scaled pullback metric $g_{c,\Phi}$ on
the Euclidean annulus, $\mathcal{A}$, as 
defined in \eqref{scaledpullback:defn}--\eqref{pullbackformula}. 
Then there is an $s_0>0$ which is 
independent of the choice of our geometrostatic manifold 
satisfying the conditions above such that
for any smooth connected $g_{c,\Phi}$-minimal surface 
$\Sigma$ in $\mathcal{A}$
with  
\be
\partial \Sigma \subseteq \partial B_\delta(0, 2) \cup \partial B_\delta(0, 1)
\ee
that contain points
\be
q \in \partial \Sigma \,\cap \, \partial B_\delta(0, 2)
\textrm{ and } q' \in \partial \Sigma \,\cap \, \partial B_\delta(0, 1)
\ee
must satisfy  
\be \label{CMcoreq}
\mathrm{Area}_{g_{c, \Phi}}(\Sigma)\geq (\pi e^{-2}) s_0^2.
\ee
\end{thm}

\begin{figure}[h] 
\begin{tikzpicture}[scale=0.8]
\draw[fill=lightgray] (0.92, 0.3) to [out=-100, in =100] (0.87, -0.7) to [out=-30, in= -120] (3.75,-1.5) to [out=100, in=-120] (4,0) to [out=150, in=0] (0.92, 0.3);
\draw[thick, red] (-1,2.6) to [out=-30, in=90] (1,0) to [out=-90, in=30] (-1,-2.6);
\draw[thick, blue] (1.92,3.92) to [out=-45, in=90] (4,0) to [out=-90, in=60] (3.1,-3);
\node[left] at (2.1, 3.4) {$\textcolor{blue}{\partial B_\delta(0, 2)}$};
\node[left] at (-0.3, 1.9) {$\textcolor{red}{\partial B_\delta(0, 1)}$};
\draw[red, thick] (0.92, 0.3) to [out=-100, in =100] (0.87, -0.7);
\draw[blue, thick] (4, 0) to [out=-120, in =100] (3.75, -1.5);
\node at (2.3, -0.5) {$\Sigma$};
\node at (-2, -0.5) {$0$};
\draw[fill] (-2,0) circle(0.08);
\node at (3.3, -1.5) {$q$};
\draw[fill] (3.75, -1.5) circle(0.08);
\node at (1.3, 0.7) {$q'$};
\draw[fill] (0.92, 0.3) circle(0.08);
\end{tikzpicture}
   \caption{The minimal surface, $\Sigma$, of Theorem~\ref{rmrk-s_0}.}
   \label{fig-S-new}
\end{figure}
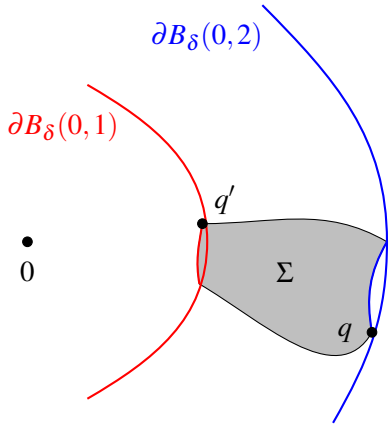

Before we prove the theorem we recall the following
 monotonicity formula (7.5) from Colding \& Minicozzi's 
textbook \cite{CoMinCourse}. 

\begin{thm}[Monotonicity Formula]\label{CM-JJ:thm}
Let $x_0$ be a point on a smooth minimal surface $\Sigma$ in a $3$-manifold $(M,g)$. Let $\kappa>0$ be a bound on sectional curvatures $K_M$ on $M$ (as in $|K_M|<\kappa$) and let $i_0>0$ denote a positive lower bound on the injectivity radius on $M$. Then the function  
\be
e^{2\sqrt{\kappa}s} s^{-2} \mathrm{Area}(B_{s,g}(x_0)\,\cap \, \Sigma)
\ee
of $0<s<\min\{i_0,\tfrac{1}{\sqrt{\kappa}}, \mathrm{dist}(x_0, \partial \Sigma)\}$ is non-decreasing. 
\end{thm}

For smooth minimal surfaces the function in Theorem \ref{CM-JJ:thm} converges to $\pi$ as $s\to 0$. Consequently, the monotonicity formula gives us an inequality of the form 
\begin{equation}\label{lowerbound:minsurf}
\mathrm{Area}(B_{s,g}(x_0)\,\cap \, \Sigma)\geq (\pi e^{-2\sqrt{\kappa}s}) s^2\geq (\pi e^{-2}) s^2
\end{equation}
on the interval for $s$ stated in the theorem. 

Following an idea used by Jauregui in \cite{Jaur-lower}, we can use this 
monotonicity formula
to provide a lower bound for the area of a minimal surface which runs between two spheres as follows:

\begin{prop}\label{CMcor}
Let $g\ge \delta$ be a metric on $B_\delta(0, 4)\setminus B_\delta(0, \tfrac{1}{2})$ whose sectional curvatures are bounded by $\kappa$ and whose injectivity radii on $B_\delta(0, 2) \cup \partial B_\delta(0, 1)$ are bounded from below by $i_0$. Let $s_0=\tfrac{1}{2}\min\{i_0,\tfrac{1}{\sqrt{\kappa}},\tfrac{1}{3}\}$. 
Then all smooth connected minimal surfaces $\Sigma$ with  
\be
\partial \Sigma \subseteq \partial B_\delta(0, 2) \cup \partial B_\delta(0, 1)
\ee
that contain points
\be
q \in \partial \Sigma \,\cap \, \partial B_\delta(0, 2)
\textrm{ and } q' \in \partial \Sigma \,\cap \, \partial B_\delta(0, 1)
\ee
must satisfy  
\be 
\mathrm{Area}_g(\Sigma)\geq (\pi e^{-2}) s_0^2.
\ee
\end{prop}

\begin{proof}
Take $x_0 \in \Sigma \,\cap \, \partial B_\delta(0,3/2)$ which exists because
$\Sigma$ is connected and has points $q$ and $q'$ as in the hypothesis.
Since $g \ge \delta$
\be
B_g(x_0, 1/2) \subset B_\delta(x_0, 1/2) \subset B_\delta(0, 1) \setminus B_\delta(0,2).
\ee
Since $s_0<1/2$ and satisfies the given bounds depending on
injectivity radius and sectional curvature, we can apply
Theorem~\ref{CM-JJ:thm} and (\ref{lowerbound:minsurf}) to obtain
\be
\mathrm{Area}_g(\Sigma)\geq \mathrm{Area}_g(\Sigma \,\cap \, B_g(x_0, s_0))
 \ge (\pi e^{-2}) s_0^2. \qedhere
\ee
\end{proof}

We now prove Theorem~\ref{rmrk-s_0}:

\begin{proof}[Proof of Theorem~\ref{rmrk-s_0}]
Note that $c$ in Theorem~\ref{rmrk-s_0}
satisfies (\ref{choice-c-ann}) which is the hypothesis of Propositions \ref{curvlemma:oct} and \ref{injrad:march}.  So we obtain uniform
bounds on $i_0$ and $\kappa$ for all $(M,g)$. 
Thus the value of $s_0$ in Proposition~\ref{CMcor} 
applied to the metric $g_{c,\Phi_0}=c^{-2}\Phi_0^*g$
does not depend on $(M,g)$ satisfying the hypotheses of the theorem.
\end{proof}

\subsection{Part of the Minimal Surface is near the Point}\label{part}
Before we prove Theorem~\ref{minsurf} we prove that every outermost
minimal surface intersects with a small ball about one of the $p_i \in P$.
Lemma~\ref{old-lem-for-prop} is applied to show each $\Sigma_i$ for $i=1..n$ intersects with a small ball about $p_i$.  Lemma~\ref{lem-for-prop} can be applied to every outermost minimal surface, even the ones which do not surround a $p_i$.

\begin{lem} \label{old-lem-for-prop}
In any Riemannian manifold, $M\subseteq \mathbb{R}^3$, endowed with a metric $g \ge \delta$, a surface $\Sigma'=\partial \Omega$ surrounding a point $p\in \Omega$ satisfies
\be 
\Sigma'\,\cap \, \bar{B}_\delta\left(p, \sqrt{\area_g(\Sigma')/(4\pi)}\,\right) \neq \emptyset.
\ee 
\end{lem}

\begin{proof}
Indeed, had there existed a $\nu>1$ such that
\be \label{BinOm}
\bar{B}_\delta\left(p, \nu\sqrt{\area_g(\Sigma')/(4\pi)}\,\right) \subset \Omega.
\ee 
then we would have
\begin{eqnarray} 
\qquad \area_g(\Sigma')&\ge &\area_\delta(\Sigma') 
\textrm{ because } g\ge \delta, \\
 &\ge& C_{iso} \left(\mathrm{Vol}_\delta(\Omega)\right)^{2/3}
\textrm{ by the isoperimetric inequality,}\\
&\ge& C_{iso} \left(\mathrm{Vol}_\delta\left(
B_\delta\left(p, \nu\sqrt{\area_g(\Sigma')/(4\pi)}\,\right)\right)\right)^{2/3} 
\textrm{ by (\ref{BinOm}),}
\\
&=&\area_\delta\left(\partial B_\delta\left(p, \nu\sqrt{\area_g(\Sigma')/(4\pi)}\,\right)\right)\\
&=& 4\pi \left(\nu\sqrt{\area_g(\Sigma')/(4\pi)}\,\right)^2
= \nu^2 \area_g(\Sigma')>\area_g(\Sigma'),
\end{eqnarray}
which is a contradiction.
\end{proof}

\begin{rmrk}\label{oldlemma}
By taking $p=p_i$ we see that in the cases of the outermost minimal surfaces $\Sigma_i$ surrounding $p_i$ we have 
\begin{equation}
\Sigma_i\,\cap \, \bar{B}_\delta\left(p_i, \sqrt{\area_g(\Sigma_i)/(4\pi)}\right) \neq \emptyset.
\end{equation}
\end{rmrk}

As it is possible there are other outermost minimizing surfaces which do not contain a point $p_i$, we prove the following lemma using the area lower bounds developed in Proposition \ref{prop-area-below}:

\begin{lem}\label{lem-for-prop}
Let $\Sigma_i$ be an outermost minimal surface of a geometrostatic manifold, $-n'\le i\le n$. There exists $j=j(i) \in \{1,...,n\}$ and $p_j\in P$ with 
\begin{equation}\label{newprop3.15}
\Sigma_i \,\cap \, B_\delta \left(p_j, \sqrt{\area_g(\Sigma_i)/\pi}\,\right)\neq \emptyset.
\end{equation} 
Note that $j(i)=i$ when $i\ge 1$.
\end{lem}

\begin{proof}
Suppose the opposite: that for all $p_j\in P$ we have 
\be
\Sigma_i\,\cap \, B_\delta\left(p_j, \sqrt{\area_g(\Sigma_i)/\pi}\right)=\emptyset.
\ee 
This is true iff for all $p_j\in P$ we have
\be \label{est-june17}
\frac{1}{|x-p_j|}\,\,\,\le\,\,\, \sqrt{\frac{\pi}{\area(\Sigma_i)}}
\ee
for all $x\in \Sigma_i$.
By the work of Huisken-Ilmanen $\Sigma_i$ is diffeomorphic to $S^2$, and so  Proposition \ref{prop-area-below} implies
\begin{equation}\label{arealowerbd}
\sqrt{\frac{4\pi} {\area_g(\Sigma_i)}} 
\,\,\,\le\,\,\, \,\max_{x\in \Sigma_i}\left\{\frac{\,|\nabla \Psi_x|\,}{\Psi^2(x)}\right\},
\end{equation}
where $\Psi(x)=\chi(x)\psi(x)$. We proceed by estimating 
${|\nabla \Psi_x|}/{\Psi^2(x)}$ using \eqref{est-june17}:
\begin{align}
\frac{\,|\nabla \Psi_x|\,}{\Psi^2(x)}\,\,\,\le 
& \,\,\,\,\frac{\,|\nabla \Psi_x|\,}{\Psi(x)}
\,\,\le\,\,\frac{\,|\nabla \chi_x|\,}{\chi(x)}+ \frac{\,|\nabla \psi_x|\,}{\psi(x)}\\
\le&\,\,\,\,\frac{\,\left(\sum_j \frac{\alpha_j}{|x-p_j|^2}\right)\,}{\chi(x)}
\,\,+\,\,\frac{\,\left(\sum_j \frac{\beta_j}{|x-p_j|^2}\right)\,}{\psi(x)}\\
\le&\,\,\,\,\sum_j\frac{1}{|x-p_j|}
\left(\frac{\left(\frac{\alpha_j}{|x-p_j|}\right)}{\chi(x)}\,+\,\frac{\left(\frac{\beta_j}{|x-p_j|}\right)}{\psi(x)}\right)\\
\le &\,\,\,\,\sqrt{\frac{\pi}{\area_g(\Sigma_i)}}\,
\left(\frac{\left(\sum_j\frac{\alpha_j}{|x-p_j|}\right)}{\chi(x)}
+\frac{\left(\sum_j\frac{\beta_j}{|x-p_j|}\right)}{\psi(x)}\right)\\
= & \,\,\,\,\sqrt{\frac{\pi}{\area_g(\Sigma_i)}}\,\left(\frac{(\chi(x)-1)}{\chi(x)}
+\frac{(\psi(x)-1)}{\psi(x)}\right)\,\,<\,\, 2 \sqrt{\frac{\pi}{\area_g(\Sigma_i)}}.
\end{align}
This chain of inequalities proves that 
\be
 \,\max_{x\in \Sigma_i}\left\{\frac{|\nabla \Psi_x|}{\Psi^2(x)}\right\}<\sqrt{\frac{4\pi} {\area_g(\Sigma_i)}}
 \ee
 which is a direct contradiction to \eqref{arealowerbd}.
 The final note follows from Lemma~\ref{old-lem-for-prop}.
\end{proof}

\subsection{The Whole Minimal Surface is Near the Point}\label{whole}
In this subsection we prove the first part of Theorem~\ref{minsurf}. Recall that $\Sigma_i$ for $-n'\le i\le n$ denotes the $i^{th}$ outermost minimal surface, and recall that 
\be
\sigma_j =\min_{k\neq j}\{|p_k-p_j|\},\ \  \sigma=\min_j\{\sigma_j, |p_j|\} \textrm{ and }
\sigma_j^{out}=\,\max_{k\neq j}\{|p_k-p_j|\}.
\ee

\begin{prop}\label{prop1}
Let $s_0$ be as in the Theorem~\ref{rmrk-s_0}.
The universal constant 
\be\label{defnC1}
C_1=1 + 2e/s_0
\ee 
is such that 
for all geometrostatic $(\mathbb{R}^3\setminus P,g)$ with $\sigma >20 mC_1$ and all $-n'\le i\le n$ there is a $j=j(i) \in \{1,...,n\}$ and $p_j\in P$ with 
\be \label{claim1a}
\Sigma_i \subseteq B_\delta\left(p_j, 2C_1\sqrt{\area_g(\Sigma_i)/\pi}\right).
\ee 
Furthermore, if $i>0$ then we may take $j(i)=i$.
\end{prop}

This proposition is a consequence of the 
following lemma which will be applied again later in the paper as well:

\begin{lem}\label{lem-prop1}
Let $s_0$ be as in the Theorem~\ref{rmrk-s_0}, let $C_1=1 + 2e/s_0$.
\begin{itemize}
\item {\sc Case 1:} Suppose $1\le i\le n$ and suppose that $r\ge \sqrt{\area_g(\Sigma_i)/ \pi}$ is a radius such that 
\be 
\sigma_i > 5 C_1 r  \text{\ \ or\ \ } \sigma_i^{out} < \tfrac{1}{5} C_1 r.
\ee
We then have 
\be \label{claim1b+}
\Sigma_i \subseteq B_\delta(p_i, 2 C_1 r).
\ee 
\item {\sc Case 2:} Suppose that $-n'\le i\le 0$ and suppose that $r\ge \sqrt{\area_g(\Sigma_i)/ \pi}$ is a radius such that 
\be 
\sigma > 5 C_1 r.
\ee
Then there is a $j=j(i) \in \{1,...,n\}$ and $p_j\in P$ for which 
\be \label{claim1b-}
\Sigma_i \subseteq B_\delta(p_j, 2 C_1 r).
\ee 
\end{itemize}
\end{lem}

Before proving Lemma~\ref{lem-prop1}, we apply it to 
prove Proposition~\ref{prop1}:

\begin{proof}[Proof of Proposition \ref{prop1}]
We are given $\sigma > 20 m C_1$.
Let $r = \sqrt{\area_g(\Sigma_i)/\pi}$.   Then by the Penrose inequality:
\be
\sigma >20 C_1m\ge 20 C_1\sqrt{\area_g(\Sigma_i)/(16\pi)}= 5C_1 r.
\ee 
Such an $r$ satisfies $r< \sigma/(5C_1)$, so we have (\ref{claim1b+}) which implies
(\ref{claim1a}).
\end{proof}

We now prove Lemma~\ref{lem-prop1}:

\begin{proof}[Proof of Lemma \ref{lem-prop1}]
Let $j=j(i)$ be as in Remark \eqref{oldlemma} and Lemma \ref{lem-for-prop}. 
Suppose the opposite: there exists a point,
\be
q_i \in \Sigma_i \setminus B_\delta(p_j, 2C_1r).
\ee

Applying Lemma~\ref{lem-for-prop}
we are able to conclude that $\Sigma_i$ contains a point  
\be
q_i\in B_\delta\left(p_j, C_1\sqrt{\area_g(\Sigma_i)/\pi}\right) \subseteq B_\delta(p_j, C_1r).   
\ee
Since $\Sigma_i$ is closed and connected, this means
we can choose $q_i$ above such that
 \be
q_i \in \Sigma_i \,\cap \,  \partial B_\delta(p_j, 2C_1r).
\ee
and we can choose 
 \be
q'_i \in \Sigma_i \,\cap \,  \partial B_\delta(p_j, C_1r).
\ee
In particular the minimal surface depicted in Figure~\ref{fig-S},
\be \label{S-as-in-fig}
S_i=\Sigma_i\,\cap \, \left(\bar{B}_\delta(p_j, 2C_1r)\setminus B_\delta(p_j, C_1r)\right),
\ee
contains the points, $q_i$ and $q'_i$, as above.  

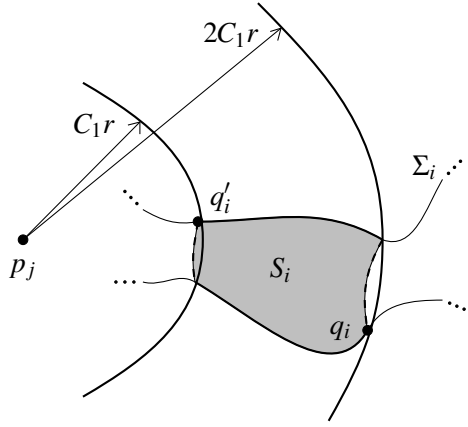
\begin{figure}[h] 
\begin{tikzpicture}[scale=0.8]
\draw[fill=lightgray] (0.92, 0.3) to [out=-100, in =100] (0.87, -0.7) to [out=-30, in= -120] (3.75,-1.5) to [out=100, in=-120] (4,0) to [out=150, in=0] (0.92, 0.3);
\draw[thick] (-1,2.6) to [out=-30, in=90] (1,0) to [out=-90, in=30] (-1,-2.6);
\draw[thick] (1.92,3.92) to [out=-45, in=90] (4,0) to [out=-90, in=60] (3.1,-3);
\draw (-2,0) to (-0.05, 1.95);
\draw (-2,0) to (2.305, 3.505); 
\draw (2.125, 3.45) to (2.305, 3.505) to (2.25, 3.3);
\node[left] at (2.1, 3.4) {$2C_1r$};
\draw (-0.25, 1.9) to (-0.05, 1.95) to (-0.05, 1.75);
\node[left] at (-0.3, 1.9) {$C_1r$};
\draw (0,0.5) to [out=-30, in =-170] (0.92, 0.3);
\draw[thick] (0.92, 0.3) to [out=0, in =150] (4,0);
\draw (4,0) to  [out =-20, in=-120] (5, 1);
\draw (0,-0.7) to [out=-15, in=150] (0.87, -0.7);
\draw[thick] (0.87, -0.7) to [out=-30, in= -120] (3.75,-1.5);
\draw (3.75, -1.5) to [out=60, in =180] (5, -1);
\draw[dashed, thick] (0.92, 0.3) to [out=-100, in =100] (0.87, -0.7);
\draw[dashed, thick] (4, 0) to [out=-120, in =100] (3.75, -1.5);
\node at (2.3, -0.5) {$S_i$};
\node at (-2, -0.5) {$p_j$};
\draw[fill] (-2,0) circle(0.08);
\node at (3.3, -1.5) {$q_i$};
\draw[fill] (3.75, -1.5) circle(0.08);
\node at (1.3, 0.7) {$q_i'$};
\draw[fill] (0.92, 0.3) circle(0.08);
\node at (4.7, 1.2) {$\Sigma_i$};
\draw[fill] (5.1, 1.1) circle (0.02);
\draw[fill] (5.2, 1.2) circle (0.02);
\draw[fill] (5.3, 1.3) circle (0.02);
\draw[fill] (5.12, -1.05) circle (0.02);
\draw[fill] (5.24, -1.1) circle (0.02);
\draw[fill] (5.36, -1.15) circle (0.02);
\draw[fill] (-0.1, 0.6) circle (0.02);
\draw[fill] (-0.2, 0.7) circle (0.02);
\draw[fill] (-0.3, 0.8) circle (0.02);
\draw[fill] (-0.15, -0.7) circle (0.02);
\draw[fill] (-0.3, -0.7) circle (0.02);
\draw[fill] (-0.45, -0.7) circle (0.02);
\end{tikzpicture}
   \caption{The minimal surface, $S_i\subset \Sigma_i$.}
   \label{fig-S}
\end{figure}

Consider the embedding 
\be 
\Phi: \left(B_\delta(0, 4)\setminus B_\delta(0, \tfrac{1}{4})\right)\to \mathbb{R}^3\setminus P \text{\ \  given by\ \ } u\mapsto p_j+ c_1 u 
\ee
where $c_1=C_1r$.

The surface $\Phi^{-1}(S_i)$ is minimal with respect to $g_{c_1,\Phi}$ and has
 points
 \be
\Phi^{-1}(q_i) \in \Phi^{-1}(S_i) \,\cap \,  \Phi^{-1}(\partial B_\delta(p_i, 2C_1r))
= \Phi^{-1}(S_i) \,\cap \, \partial B_\delta(0, 2)
\ee
and 
 \be
\Phi^{-1}(q'_i) \in \Phi^{-1}(S_i) \,\cap \, \Phi^{-1}( \partial B_\delta(p_i, C_1r)) = \Phi^{-1}(S_i) \,\cap \, \partial B_\delta(0, 1).
\ee
In addition,
\be 
\partial \Phi^{-1}(S_i) 
\subset \partial B_\delta(0, 2) \cup \partial B_\delta(0, 1).
\ee
We may now apply Theorem~\ref{rmrk-s_0} 
to obtain
\be \label{CMcoreq-again}
\mathrm{Area}_{g_{c_1,\Phi}}(\Phi^{-1}(S_i))\ge \pi e^{-2} s_0^2.
\ee
Thus
\begin{eqnarray} 
\qquad \area_g(\Sigma_i) &\ge& \mathrm{Area}_{g}(S_i) 
\quad\qquad\qquad \qquad \textrm{ by } S_i \subset \Sigma_i\\
&=&\mathrm{Area}_{\Phi^*g}(\Phi^{-1}(S_i)) 
\qquad\qquad\textrm{ by the defn of pullback}\\
&=& c_1^2 \mathrm{Area}_{g_{c_1,\Phi}}(\Phi^{-1}(S_i)) 
\qquad\quad\textrm{ by rescaling}\\
&\ge & c_1^2 \pi e^{-2} s_0^2 
\qquad\qquad \quad\qquad \,\,\textrm{ by (\ref{CMcoreq}) }\\
&=& \left(C_1x\right)^2 \pi e^{-2} s_0^2 
\qquad\qquad\quad \textrm{ by $c_1=C_1r$}\\
&\ge & s_0^2 C_1^2\area_g(\Sigma_i)/e^2 
\qquad\qquad \quad\,\,\textrm{ by $r \ge \sqrt{\area_g(\Sigma_i)/\pi}$ }\\
&>& \area_g(\Sigma_i) \qquad \qquad
\qquad\qquad\qquad\textrm{ because $C_1>e/s_0$},
\end{eqnarray}
which is a contradiction. 
\end{proof}

\subsection{The Minimal Surface is Not Too Close to the Point}\label{not-close}

We now prove the second part of Theorem~\ref{minsurf}.
Recall $C_1$ defined in Proposition~\ref{prop1}.

\begin{prop}\label{prop2}
If $m<\sigma/(20C_1)$, then 
\be \label{claim2}
\Sigma_i \,\cap \, B_\delta\left(p_i, \frac{\alpha_i\beta_i}{4C_1(\alpha_i+\beta_i)}\right)= \emptyset \text{\ \ for all\ \ }1\le i\le n.
\ee  
\end{prop}

In order to prove Proposition \ref{prop2} we apply an inversion to the geometrostatic manifold sending $p_i$ to $\infty$.   As we could not find this inversion process in the literature, we provide the details in the appendix.  See Theorem~\ref{inversion}.

\begin{proof}[Proof of Proposition \ref{prop2}]
Without loss of generality take $i=n$.
Apply the inversion of Theorem~\ref{inversion} with $x_i=p_i$,
to obtain a minimal surface $F^{-1}(\Sigma_n)$ about $y_0=0$ in 
\be
Y=\mathbb{R}^n\setminus\{y_0,...,y_{n-1}\} 
\textrm{ with }
g_Y=\left(1+\sum_{i=0}^{n-1} \frac{\alpha_{Y,i}}{|y-y_i|}\right)^2
\left(1+\sum_{i=0}^{n-1} \frac{\beta_{Y,i}}{|y-y_i|}\right)^2 \delta_y.
\ee
There is an outermost minimal surface about $y_0=0$ such that
\be
\Sigma_0=\partial \Omega_0 \subset Y.
\ee
By the definition of outermost as in Section~\ref{sect-HI}, we know  
\be\label{FinU}
F^{-1}(\Sigma_n) \subset \Omega_0.
\ee

Let
\be
A_0=\area_{g_Y}(\Sigma_0)
\ee
In Corollary~\ref{mADM-Y} it was seen that the
ADM mass of $(Y, g_Y)$ is 
\be
m_Y=\alpha_n +\beta_n +\sum_{j=1}^{n-1} \frac{\alpha_n\beta_{j}+\alpha_{j}\beta_n}{r_{j,n}}=m_n.
\ee
Observe that  
\be\label{mnanbn}
m_n \le \alpha_n+\beta_n + \alpha_n\sum_{j=1}^{n-1} 
\tfrac{\beta_j}{\sigma}
+ \beta_n\sum_{j=1}^{n-1} \tfrac{\alpha_j}{\sigma}
<(\alpha_n+\beta_n)(1+\tfrac{m}{\sigma})
\le 2(\alpha_n+\beta_n)
\ee
By the Penrose Inequality, and the fact that $\Sigma_0$ is outermost
minimizing in $(Y, g_Y)$ we have
\be
m_n=m_Y \ge \sqrt{A_0/(16\pi)}.
\ee
In addition, by Theorem~\ref{inversion},
 the $0^{th}$ separation constant of $(Y, g_Y)$
satisfies
\be\label{x-check-Penrose}
\sigma^{out}_0= \,\max_{j\neq 0}\{|y_j-0|\}
=\,\max_{j\neq 0} \{\alpha_n \beta_n /r_{j,n} \} \le \alpha_n \beta_n/ \sigma
\ee
where $\sigma$ is the separation constant for $M$ as in
(\ref{separation}).
Furthermore
\be
\frac{\alpha_n \beta_n}{ \sigma} =
\frac{\alpha_n}{(\alpha_n+\beta_n)} \frac{\beta_n}{\sigma} (\alpha_n+\beta_n)
\le 1 \cdot \frac{m}{\sigma} (\alpha_n+\beta_n) \le \frac{1}{5} (\alpha_n+\beta_n).
\ee
In particular
\be\label{x-check-1}
\sigma^{out}_0 \le m_n/5.
\ee
We may apply Lemma~\ref{lem-prop1}  to
the geometrostatic manifold $(Y, g_Y)$ with $i=0$ taking
$r=m_n$ by (\ref{x-check-Penrose}) and (\ref{x-check-1}).
Thus 
\be \label{claim1}
\Sigma_0 \subseteq B_\delta(0, 2 C_1 r).
\ee 
As a consequence, by (\ref{FinU}), we have
\be
F^{-1}(\Sigma_n) \subset \Omega_0 \subseteq B_\delta(0, 2 C_1 r).
\ee
Thus
\be
\Sigma_n \subset F(B_\delta(0, 2 C_1 r))= \mathbb{R}^3 \setminus
\bar{B}_\delta\left(p_n, \tfrac{\alpha_n\beta_n}{2C_1r}\right)
\ee
By our choice of $r=m_n$ and (\ref{mnanbn}) we have
\be
r < 2(\alpha_n +\beta_n)
\ee
which implies that
\be
\frac{\alpha_n\beta_n}{2C_1r}>\frac{\alpha_n\beta_n}{4C_1(\alpha_n+\beta_n)}.
\ee
Thus
\be
\Sigma_n \,\cap \, B_\delta\left(p_n, \tfrac{\alpha_n\beta_n}{4C_1(\alpha_n+\beta_n)}\right)
\subset
\Sigma_n \,\cap \, B_\delta\left(p_n, \tfrac{\alpha_n\beta_n}{2C_1r}\right)=\emptyset,
\ee
and our proof is now complete.
\end{proof}

\subsection{Proof of Theorems~\ref{minsurf} and~\ref{thm-M'} } 
Theorem \ref{minsurf} is a special case of the following, more general result. The result shows that all the outermost minimal surfaces $\Sigma_i$ with $i\in \{-n',..., n\}$ are located in an annular neighborhood of some point $p_{j(i)}\in P$. 

\begin{thm}\label{Ivascor}
Let $s_0$ be as in the Theorem~\ref{rmrk-s_0}.
The universal constant 
\be\label{defnC1x}
C_1=1 + 2e/s_0
\ee 
is such that for all geometrostatic $(\mathbb{R}^3\setminus P,g)$ with $\sigma >20 mC_1$ and all $-n'\le i\le n$ there is a $j=j(i) \in \{1,...,n\}$ and $p_j\in P$ with 
\be \label{claim1ax}
\Sigma_i \subseteq B_\delta\left(p_j, 2C_1\sqrt{\area_g(\Sigma_i)/\pi}\right)\setminus B_\delta\left(p_j, \tfrac{\alpha_j\beta_j}{4C_1(\alpha_j+\beta_j)}\right).
\ee 
Furthermore, if $i>0$ then we may take $j(i)=i$.
\end{thm}

\begin{proof}
Note that for all $i\in \{-n',..., n\}$ there is a  $j=j(i) \in \{1,...,n\}$ and $p_j\in P$ with
\be
\Sigma_i \subseteq B_\delta\left(p_j, 2C_1\sqrt{\area_g(\Sigma_i)/\pi}\right)
\ee
by Proposition \ref{prop1}. Furthermore for $j\in \{1,...,n\}$ by Proposition \ref{prop2} we have
\be 
\Sigma_j \,\cap \, B_\delta\left(p_j, \tfrac{\alpha_j\beta_j}{4C_1(\alpha_j+\beta_j)}\right) = \emptyset
\ee 
so 
\be \label{out-of-ball}
B_\delta\left(p_j, \tfrac{\alpha_j\beta_j}{4C_1(\alpha_j+\beta_j)}\right) \subset \Omega_j.
\ee
We see from \eqref{disjointOmega} that $\Sigma_j \,\cap \, \Omega_{j(i)} = \emptyset$ for all $i\in \{-n',...,0\}$.  Combining this observation with (\ref{out-of-ball}) yields
\be 
\Sigma_i \,\cap \, B_\delta\left(p_j, \tfrac{\alpha_j\beta_j}{4C_1(\alpha_j+\beta_j)}\right)= \Sigma_i \,\cap \, \Omega_j =\emptyset,
 \ee 
which completes our proof. 
\end{proof}

For $j\in \{1, ..., n\}$ define 
\be
I(j):=\left\{i\ \big{|}\ -n'\le i\le n,\ \Sigma_i\subseteq B_\delta\left(p_j, 2C_1\sqrt{\area_g(\Sigma_i)/\pi}\right)\right\}.
\ee
Loosely speaking, the set $I(j)$ identifies those outermost minimal surfaces which are close to one of $p_j\in P$. 
Next, introduce 
\be \label{defn-gammaj}
\gamma_j:=\,\max\left\{2C_1\sqrt{\area_g(\Sigma_i)/\pi}\ \big{|}\  i\in I(j)\right\}
\ee 
and note that  
\be\label{randomPenrose}
\gamma_j\le 8C_1m. 
\ee
by the Penrose inequality. The lengths $\gamma_j$ are the radii within which all the outermost minimal surfaces are to be found. With this notation established, Theorem \ref{thm-M'} is an immediate corollary of Theorem \ref{Ivascor}.

\section{Almost Rigidity of the Positive Mass Theorem}
\label{sect-almost-rigidity}

In this section we prove the Almost Rigidity of the Positive Mass
Theorem for geometrostatic manifolds [Theorem~\ref{Almost-Rigidity}]. Observe that our result includes (but is not limited to) geometrostatic manifolds with a uniform upper bound on the number of black holes whose ADM mass is converging to $0$. 

\subsection{A Review of Intrinsic Flat Convergence}

The intrinsic flat distance, $d_{\mathcal{F}}$ between pairs of 
compact oriented Riemannian manifolds
with boundary was first introduced by the first author with Wenger in
\cite{SW-JDG}.  Their intrinsic flat distance, like the classical flat distance
of Geometric Measure Theory, does not scale well: the distance between two
$n$ dimensional oriented manfolds is a sum of an $(n+1)$ dimensional
filling volume and an $n$ dimensional volume.  In joint work of the first author with LeFloch \cite{LS}, the D-flat distance
$d_{D\mathcal{F}}$ was defined, dividing the $(n+1)$ dimensional volume by diameter before adding it the $n$ dimensional volume.  

In work of Lakzian and the first author \cite{Lakzian-Sormani}
the following theorem was proven providing a concrete means to estimate the intrinsic flat distance.  Here we state it also adding in the estimate on the
D-flat distance multiplied by diameter:

\begin{thm} \label{thm-subdiffeo} 
Suppose $(M_1,g_1)$ and $(M_2,g_2)$ are oriented
precompact Riemannian manifolds
with diffeomorphic subregions $W_i \subset M_i$.
Identifying $W_1=W_2=W$ assume that
on $W$ we have
\be \label{thm-subdiffeo-1}
g_1 \le (1+\varepsilon)^2 g_2 \textrm{ and }
g_2 \le (1+\varepsilon)^2 g_1.
\ee
Taking the extrinsic diameters,
\be \label{DU}
\diam(M_i) \le D
\ee
we define a hemispherical width,
\be \label{thm-subdiffeo-3}
a>\frac{\arccos(1+\varepsilon)^{-1} }{\pi}D.
\ee
Taking the difference in distances with respect to the outside manifolds,
we set
\be \label{lambda}
\lambda=\sup_{x,y \in W}
|d_{M_1}(x,y)-d_{M_2}(x,y)| \le 2D,
\ee
and we define the height,
\be \label{thm-subdiffeo-5}
\bar{h}=\,\max\{\sqrt{2\lambda D},  D\sqrt{\e^2 + 2\e} \}.
\ee
Then
\begin{align*}
d_{\mathcal{F}}(M_1, M_2) &\le
\left(2\bar{h} + a\right) \Big(
\vol_m(W_{1})+\vol_m(W_2)+\vol_{m-1}(\partial W_{1})+\vol_{m-1}(\partial W_{2})\Big)\\
&\qquad+\vol_m(M_1\setminus W_1)+\vol_m(M_2\setminus W_2),
\end {align*}
and
\begin{align*}
Dd_{D\mathcal{F}}(M_1, M_2) &\le
\left(2\bar{h} + a\right) \Big(
\vol_m(W_{1})+\vol_m(W_2)+D\vol_{m-1}(\partial W_{1})+D\vol_{m-1}(\partial W_{2})\Big)\\
&\qquad +\vol_m(M_1\setminus W_1)D+\vol_m(M_2\setminus W_2)D.
\end {align*}
\end{thm}

\subsection{Our strategy}
We start by fixing a geometrostatic manifold $(M,g)$ and a value of $R>0$ such that $|p_i|\neq R$ for all $1\le i\le n$. We let 
\be \label{defn-M_1}
M_1=\{x\in M'\big{|}\, d_{(M',g)}(0,x)<R\}=B_g(0,R)\subseteq M'
\ee
endowed with the distance $d_{(M',g)}$
and
\be \label{defn-M_2}
M_2=\{x\in \mathbb{R}^3\big{|}\, |x|<R\}= B_\delta(0,R)
\ee 
endowed with the distance $d_{(\mathbb{R}^3, d_\delta)}(x,y)=|x-y|$.
In the next few sections we prove estimates which allow us to apply 
Theorem \ref{thm-subdiffeo}. Ultimately, we obtain the following bound on the intrinsic flat distance between $M_1$ and $M_2$ with these distances defined above. 

\begin{prop}\label{main-prop}  
There exist universal constants $\e_0\in (0,1)$, $C_{\mathcal{F}}'$, $C_{\mathcal{F}}''$ and $C_{D\mathcal{F}}$ such that for all $R>0$, all $\e \in (0,\e_0)$
and all Brill-Lindquist geometrostatic manifolds $(M,g)=(\mathbb{R}^3\setminus P, (\chi\psi)^2\delta)$ with 
\be
m_{ADM}(M') = m < R \e^3,\ \ m<\e\cdot \tfrac{\sigma}{32}\ \ \text{\ \ and\ \ } \rho(P)\,\cap \, (R-32R\e, R+32R\e)=\emptyset
\ee
where $\rho(x)=|x|$ and $P=\{p_1,...,p_N\}$, 
the intrinsic flat distance is estimated by
\be\label{flatdistance-est}
d_{\mathcal{F}}(M_1, M_2) \le 
C_{\mathcal{F}}' R^4 \sqrt{\e} + C_{\mathcal{F}}'' R^3 \sqrt{\e}
\ee
and the D-flat distance is estimated by
\be\label{D-flatdistance-est}
d_{D\mathcal{F}}(M_1, M_2) \le 
C_{D\mathcal{F}} R^3 \sqrt{\e}.
\ee
\end{prop}

\begin{rmrk}
Note that in Proposition~\ref{main-prop} our estimates
do not depend upon the number of points $p_i\in P$
nor on the number of minimal surfaces $\Sigma_i$.
\end{rmrk}

The proof of Proposition \ref{main-prop} is involved and is proven
in the next few subsections. The main result, Theorem~\ref{Almost-Rigidity}, follows as a straightforward consequence of Proposition \ref{main-prop}; see subsection \ref{proof-mainthm} below for details. 

\subsection{Locating $M_1$}

It is important to understand that there is a possibility for the asymptotic ends of 
$M$ to be connected by very long almost-cylindrical regions. In other words, there is a possibility for the connected components $\Sigma_i$ of $\partial M'$ to be located very far down a deep well at $p_i$. See Example~\ref{ex-lim-ERN} where $\beta_i<<\alpha_i$.

In these settings $M_1=B_g(0,R)\subseteq M'$ not only controls $|x|$ but also cuts off any long near-cylindrical regions near $p_i$. To make this idea precise we introduce the length
\be \label{defn-delta}
\delta_{i,R}=\,\max\left\{(\alpha_i+\beta_i)\exp\left(\frac{-R}{\alpha_i+\beta_i}\right), \frac{\alpha_i\beta_i}{4C_1(\alpha_i+\beta_i)}\right\}.
\ee
When there is a long cylindrical neck then the first term achieves the maximum here.

The next lemma clarifies how long near-cylindrical regions are cut off from $M_1$. Note that the material of this subsection is independent of our choice of $\e$.

\begin{lemma}\label{locatingM1}
If $m<\sigma/(20 C_1)$ then 
\be 
M_1\subseteq B_\delta(0,R)\setminus \left(\bigcup_i B_\delta (p_i,\delta_{i,R})\right).
\ee
\end{lemma}

\begin{proof}
First observe that since $g\ge \delta$ we have
\be
M_1=B_g(0,R)\subseteq B_\delta(0,R).
\ee
So we need only show that
\be\label{deltaionlyshow}
B_g(0,R) \subseteq \mathbb{R}^3\,\setminus \left(\bigcup_i 
B_\delta\left(p_i, \delta_{i,R} \right)\right).
\ee
When $\delta_{i,R}= \frac {\alpha_i \beta_i}{4C_1(\alpha_i+\beta_i)}$ this is immediate from Theorem \ref{thm-M'}:
\be \label{eq-prop2}
M_1\,\,\subseteq \,\,M'
\,\,\subseteq\,\, \mathbb{R}^3\,\setminus \left(\bigcup_i 
B_\delta\left(p_i, \frac{\alpha_i\beta_i}{4C_1(\alpha_i+\beta_i)}\right)\right).
\ee  
In the case of a long cylindrical end, when 
$\delta_{i,R}=(\alpha_i+\beta_i)\exp\left(-{R}/{(\alpha_i+\beta_i)}\right),$
we obtain (\ref{deltaionlyshow}) by proving that
\be \label{deltaionlyshow2}
d_g\left(0,\partial B_\delta\left(p_i,\delta_{i,R}\right)\right)>R.
\ee
Using $\alpha_i+\beta_i<m<\sigma\le |p_i|$ and the fact that 
\be
g>\left(1+\tfrac{\alpha_i}{|x-p_i|}\right)^2\left(1+\tfrac{\beta_i}{|x-p_i|}\right)^2\delta
\ee
 we can compute:
\begin{eqnarray}
\quad d_g\left(0,\partial B_\delta(p_i,\delta_{i,R})\right)&>
&d_g\left(\partial B_\delta(p_i,\alpha_i+\beta_i),
\partial B_\delta(p_i,\delta_{i,R})\right)\\
&=&\int_{t=0}^1 g(\gamma'(t), \gamma'(t))^{1/2} \, dt\\
& &\qquad \textrm{ where $\gamma$ is a minimal geodesic}\\
&=&\int_{t=0}^1 \left(1+\tfrac{\alpha_i}{|\gamma(t)-p_i|}\right)
\left(1+\tfrac{\beta_i}{|\gamma(t)-p_i|}\right) 
|\gamma'(t)| \, dt\\
&\ge&\int_{t=0}^1 \left(1+\tfrac{\alpha_i}{r}\right)
\left(1+\tfrac{\beta_i}{r}\right) 
\tfrac{d}{dt} (r(\gamma(t))) \, dt\\
&\ge&\int_{\delta_{i,R}}^{\alpha_i+\beta_i} (1+\tfrac{\alpha_i}{r})(1+\tfrac{\beta_i}{r})\,dr\\
&>&\int_{\delta_{i,R}}^{\alpha_i+\beta_i} \tfrac{\alpha_i+\beta_i}{r}\,dr
>(\alpha_i+\beta_i)\ln\left(\frac{\alpha_i+\beta_i}{\delta_{i,R}}\right)
=R.
\end{eqnarray}
This gives us (\ref{deltaionlyshow2}) which implies
(\ref{deltaionlyshow}), and we are done.
\end{proof}

\subsection{Proximity of $g$ to $\delta$}

We continue by identifying a region of $M'$ where $g$ is close to $\delta$ in the sense of \eqref{thm-subdiffeo-1}. We note that the results of this subsection are independent of the parameter $R$.

\begin{lemma}\label{lemma1:oct}
Let $\e>0$ and assume that $m<\e\cdot \tfrac{\sigma}{16}$.
Then on 
\be
\mathbb{R}^3\setminus\left(\bigcup_{i=1}^n B_\delta\left(p_i, \frac{8}{\e}(\alpha_i+\beta_i)\right)\right).
\ee
we have 
\be
\delta \le g\le (1+\e)^2 \delta.
\ee
\end{lemma}
\begin{proof}
It suffices to prove that 
\be1+\sum_i \frac{\alpha_i}{|x-p_i|}<1+\e/4\ \ \text{and}\ \ 1+\sum_i\frac{\beta_i}{|x-p_i|}<1+\frac{\e}{4}
\ee
Our hypothesis on $m$ gives us
\be 
\frac{8}{\e}(\alpha_i+\beta_i)<\frac{8m}{\e}<\frac{\sigma}{2}.
\ee
Suppose $x\not\in \bigcup_i B_\delta(p_i, \tfrac{8}{\e}(\alpha_i+\beta_i))$. In the case when $x\not\in \bigcup_i B_\delta(p_i,\sigma/2)$ we have  
\be\sum_i \frac{\alpha_i}{|x-p_i|}<2\sum_i \frac{\alpha_i}{\sigma}<2\frac{m}{\sigma}<\frac{\e}{8},\ee
and an analogous inequality with $\beta_i$. On the other hand, if $|x-p_j|<\sigma/2$ for some (and hence exactly one) $j$ then 
\be
\sum_i \frac{\alpha_i}{|x-p_i|}<\frac{\e\alpha_j}{8(\alpha_j+\beta_j)}+2\sum \frac{\alpha_i}{\sigma}<\frac{\e}{8}\left(\frac{\alpha_j}{\alpha_j+\beta_j}+1\right)<\frac{\e}{4}.\ee
(An analogous inequality can be proven for $\beta$'s as well.)
\end{proof}

For a fixed $\e>0$, Brill-Lindquist geometrostatic manifolds whose ADM mass satisfies  $m<\sigma/(20 C_1)$ and $m<\e\cdot \tfrac{\sigma}{16}$, and lengths $\gamma_i$ from \eqref{defn-gammaj} we define
\be\label{def-gamma}
\gamma_{i,\e}=\,\max\left\{\frac{8}{\e}(\alpha_i+\beta_i), \gamma_i\right\}.
\ee
The purpose of introducing $\gamma_{i,\e}$ is in marking the portion of $M'$ 
\be\label{minithmM'}
{\mathbb{R}}^3 \setminus \left(\bigcup_{i=1}^n B_\delta\left(p_i, \gamma_{i,\e}\right)\right) \subseteq M' 
\ee
on which the metric $g$ is suitably close to the Euclidean metric $\delta$:
\be\label{minithmM''}
\delta \le g\le (1+\e)^2 \delta.
\ee
We now record several estimates involving $\gamma_{i,\e}$ which are needed later. 

\begin{lemma}\label{jan-add}   
Let $0<\e<\e_0:=\sqrt{ \frac {2}{\pi C_1^2} }$. We have 
\begin{enumerate}
\item $\gamma_{i,\e}\le 8m/\e$ for all $1\le i\le n$;
\medbreak
\item $\sum \gamma_{i,\e}^2<96 m^2/\e^2$;   
\medbreak
\item $\sum \gamma_{i,\e}^3< 768 m^3/\e^3$;  
\end{enumerate}
\end{lemma}

\begin{proof}
Since $\alpha_i+\beta_i<\sum\left(\alpha_i+\beta_i\right)=m$, the first claim follows from \eqref{randomPenrose}:
\be
\gamma_{i,\e}
\le 8m \,\max\{\tfrac{1}{\e},C_1\} = \tfrac{8m}{\e},
\ee
and so does the second claim,
\be \label{sigmagammasquared:est}
\begin{aligned}
\sum \gamma_{i,\e}^2\,\,\le &\,\,\tfrac{64}{\e^2} \sum m(\alpha_i+\beta_i)+C_1^2 \sum \area_g(\Sigma_i)\\
\le &\,\,\frac{64m^2}{\e^2} + 16\pi C_1^2 m^2< 96 \frac{m^2}{\e^2}.
\end{aligned}\ee
The third estimate is immediate from the first two.  
\end{proof}

\subsection{Estimating lengths}
In this subsection we estimate the length parameter $\lambda$ of \eqref{lambda}.

\begin{lemma}\label{lemma2:oct}
Assume that 
\be\label{three-bounds}
m<\e\cdot \tfrac{\sigma}{32} \textrm{ and } m<R\e^3  
\textrm{ where }0<\e<\e_0:=\sqrt{\tfrac{2}{\pi C_1^2}}.  
\ee
Furthermore, let 
\be \label{defn-W'}
W'=B_\delta(0,R)\setminus\left(\bigcup_i B_\delta(p_i, \gamma_{i,\e})\right).
\ee 
Then $W'\subseteq M'$ and the parameter 
\be
\lambda:=\sup_{x,y\in W'} |d_{(M',g)}(x,y)-d_{(\mathbb{R}^3,\delta)}(x,y)|
\ee
satisfies 
\be\label{lambda-R-e}
\lambda <\lambda_{R,\e}:= 24 R\e.
\ee
In particular, $\lambda$ scales like distance and converges to $0$ for fixed $R$ as $\e$ to $0$.
\end{lemma}

Note that this $\lambda$ will be useful for estimating the parameter (\ref{lambda}) for any set $W\subseteq W'$ as well.

\begin{proof}
It follows from (\ref{three-bounds}) that 
$\sigma> 20 m C_1$. Consequently, \eqref{minithmM'} applies and we have $W' \subseteq M'$.  Now let $x,y\in W'$. These two points 
can be joined by a path $\varphi$ in $M'$ consisting of portions of the Euclidean line segment $xy$ and interrupted by several at most semi-circular arcs along the spheres of Euclidean radius $\gamma_{i,\e}$ (for varying $i$)
by Proposition~\ref{prop1} and definitions \eqref{defn-gammaj} and \eqref{def-gamma}. The centers of these spheres project onto points on the line segment $xy$ which are at least 
\be
\sqrt{\sigma^2-\,\max_i\{2\gamma_{i,\e}\}^2}
\ge \sqrt{\sigma^2-\left(\frac{16m}{\e}\right)^2}
> \sqrt{\sigma^2-\left(\frac{\sigma}{2}\right)^2}>\frac{\sigma}{2}
\ee
away from each other; consult the diagram below for details. 

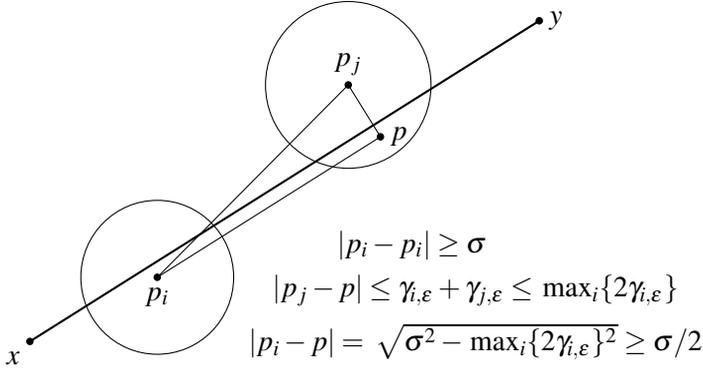
\begin{figure}[h]
\begin{tikzpicture}[scale=.85]
\draw[thick] (0,0) -- (8, 5);
\draw[fill=black] (0,0) circle (0.05);
\draw[fill=black] (8,5) circle (0.05);
\draw (2,1) circle (1.2);
\draw (5,4) circle (1.3);
\node[below left] at (0, 0) {$x$};
\node[right] at (8, 5) {$y$};
\draw[fill=black] (2,1) circle (0.05);
\draw[fill=black] (5,4) circle (0.05);
\node[below] at (2,1) {$p_i$};
\node[above] at (5,4) {$p_j$};
\draw (2,1) -- (5,4);
\draw (2,1) -- (490/89, 284/89);
\draw (5,4) -- (490/89, 284/89);
\draw[fill=black] (490/89, 284/89) circle (0.05);
\node[right] at (490/89, 284/89) {$p$};
\node at (6,1.5) {$|p_i-p_i|\ge \sigma$};
\node at (7,.8) {$|p_j-p|\le \gamma_{i,\e}+\gamma_{j,\e}\le \,\max_i\{2\gamma_{i,\e}\}$};
\node at (7,0) {$|p_i-p|=\sqrt{\sigma^2-\,\max_i\{2\gamma_{i,\e}\}^2} \ge \sigma/2$};
\end{tikzpicture}
\caption{Illustration for the proof of Lemma \ref{lemma2:oct}.}
\label{fig-lemma-4-7}
\end{figure}

It follows from $|x-y|<2R$ that there can be no more than $\frac{4R}{\sigma}$ arcs on the path $\varphi$. In particular, the length of $\varphi$ measured with respect to Euclidean metric satisfies
\be \label{sigmagamma:est}
d_{(M',\delta)}(x,y) \le L_\delta(\varphi) \le |x-y|+\pi \sum \gamma_{i,\e}
\ee
where the summation in the last line goes over at most $\frac{4R}{\sigma}$ elements. By Cauchy-Schwarz inequality we see that the latter sum satisfies 
\be
\left(\sum \gamma_{i,\e}\right)^2\le \frac{4R}{\sigma} \sum \gamma_{i,\e}^2.
\ee
It follows from Lemma \ref{jan-add} and the assumption $m<\e\cdot \tfrac{\sigma}{32}$ that 
\be 
\sum \gamma_{i,\e}^2< 96 m^2/\e^2<3\sigma m/\e.
\ee
Overall, we see that the summation term in \eqref{sigmagamma:est} can be bounded by 
\be
\left(\sum \gamma_{i,\e}\right)^2
\le \frac{4R}{\sigma}\cdot\frac{3\sigma m}{\e} \le \frac{12mR}{\e}.
\ee
Combining this with (\ref{sigmagamma:est}) we have
\be \label{|x-y|}
d_{(M',\delta)}(x,y) \le |x-y|+\pi \sum \gamma_{i,\e}
< d_{(\mathbb{R}^3,\delta)}(x,y)+\pi \sqrt{\frac{12mR}{\e}}.
\ee

By Lemma \ref{lemma1:oct} we have
\be 
d_{(\mathbb{R}^3,\delta)}(x,y)\le d_{(M',\delta)}(x,y)
\le d_{(M',g)}(x,y)\le (1+\e)d_{(M',\delta)}(x,y).
\ee
Since $m<R\e^3$ the estimate  \eqref{|x-y|} implies 
\be
d_{(\mathbb{R}^3,\delta)}(x,y)
\le d_{(M',g)}(x,y)\le(1+\e)
\left(d_{(\mathbb{R}^3,\delta)}(x,y)+\pi R \e \sqrt{12}\,\right).
\ee
It follows from $d_{(\mathbb{R}^3,\delta)}(x,y)=|x-y|<2R$ that 
\be 
0\le d_{(M',g)}(x,y)-d_{(\mathbb{R}^3,\delta)}(x,y)\le 
R\e(2 + (1+\e) \pi \sqrt{12}).
\ee
The claim \eqref{lambda-R-e} is now immediate from $2 + (1+\e) \pi \sqrt{12}< 2+4\pi \sqrt{3}<24$.
\end{proof}

\begin{rmrk}
Note that in Lemma~\ref{lemma2:oct} our estimates
do not depend upon the number of points $p_i\in P$
nor on the number pf minimal surfaces $\Sigma_i$.
\end{rmrk}

\subsection{Introducing $W$}

In order to apply Theorem \ref{thm-subdiffeo} we need a pair of diffeomorphic subregions $W_1$ and $W_2$.  We are able to simplify the situation slightly by choosing a single $W\subseteq  {\mathbb{R}}^3$
which can be viewed as both a subset of $M_1$ and $M_2$ of (\ref{defn-M_1})-(\ref{defn-M_2}).   Our $W'$ defined in \eqref{defn-W'} may not be a subset of both these manifolds.

Define
\be \label{U}
W=B_\delta(0,R-\lambda)\setminus\left(\bigcup_iB_\delta (p_i, \gamma_{i,\e})\right),
\ee
where $\lambda$ is the parameter estimated in Lemma \ref{lemma2:oct}
and $\gamma_{i,\epsilon}$ is defined in (\ref{def-gamma}).
Clearly, $W= W'\,\cap \, B_\delta(0, R-\lambda)\subseteq W'$ of Lemma \ref{lemma2:oct}. See Figure~\ref{fig-W-W'} for two different scenarios
as to how $W\subseteq M_1\subseteq M'$ depending on 
which whether the $i^{th}$ minimal surface
$\Sigma_i$ is located outside or inside
of $B_{\delta}(p_i, \delta_{i,R})$.   

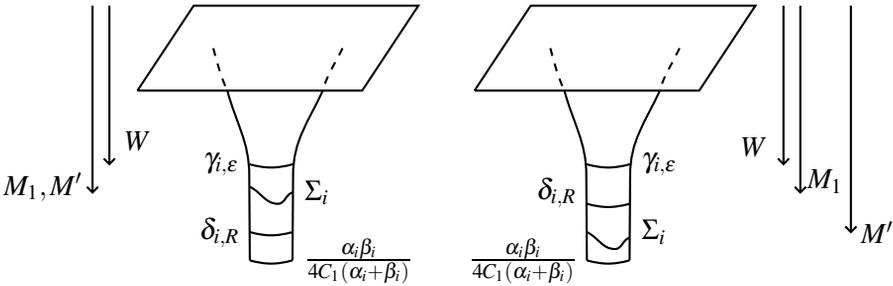
\begin{figure}[h] 
\centering
\begin{tikzpicture}[scale=.75]

\node[right] at (0.8, -8) {$\frac{\alpha_i\beta_i}{4C_1(\alpha_i+\beta_i)}$};
\node[left] at (0, -7.5) {$\delta_{i,R}$};
\node[right] at (0.8, -6.8) {$\Sigma_i$};
\node[left] at (0, -6.3) {$\gamma_{i,\e}$};

\draw[thick] (-2.5, -3.5) to (-2.5, -6.3);
\draw[thick] (-2.6, -6.2) to (-2.5, -6.3) to (-2.4, -6.2);
\node[right] at (-2.4, -5.9) {$W$};
\draw[thick] (-2.8, -3.5) to (-2.8, -6.8);
\draw[thick] (-2.9, -6.7) to (-2.8, -6.8) to (-2.7, -6.7);
\node[left] at (-2.8, -6.7) {$M_1, M'$};

\draw[thick] (-2,-5) to (2, -5) to (3,-3.5) to (-1, -3.5) to (-2, -5);
\draw[thick, dashed] (-.65, -4.25) to [out=-75, in=115] (-0.4, -5);
\draw[thick] (-.4, -5) to [out=-75, in=90] (-0.01, -6.5) to (0,-8);
\draw[thick, dashed] (1.65, -4.25) to [out=-115, in=75] (1.3, -5);
\draw[thick] (1.3, -5) to [out=-115, in=90] (0.77, -6.5) to (0.76,-8);
\draw[thick] (0,-8) to [out =-15, in=-165] (0.76, -8);
\draw[thick] (0,-7.5) to [out =-15, in=-165] (0.76, -7.5);
\draw[thick] (0, -6.7) to [out=-25, in =-180] (0.5, -7) to [out =0, in =-165] (0.76, -6.8);
\draw[thick] (-0.01, -6.3) to [out=-15, in =-165] (0.77, -6.3);

\node[left] at (6, -8) {$\frac{\alpha_i\beta_i}{4C_1(\alpha_i+\beta_i)}$};
\node[right] at (6.8, -7.5) {$\Sigma_i$};
\node[left] at (6, -6.8) {$\delta_{i,R}$};
\node[right] at (6.8, -6.3) {$\gamma_{i,\e}$};

\draw[thick] (9.5, -3.5) to (9.5, -6.3);
\draw[thick] (9.4, -6.2) to (9.5, -6.3) to (9.6, -6.2);
\node[left] at (9.4, -6) {$W$};
\draw[thick] (9.8, -3.5) to (9.8, -6.8);
\draw[thick] (9.7, -6.7) to (9.8, -6.8) to (9.9, -6.7);
\node[right] at (9.75, -6.6) {$M_1$};
\draw[thick] (10.7, -3.5) to (10.7, -7.5);
\draw[thick] (10.6, -7.4) to (10.7, -7.5) to (10.8, -7.4);
\node[right] at (10.7, -7.5) {$M'$};

\draw[thick] (4,-5) to (8, -5) to (9,-3.5) to (5, -3.5) to (4, -5);
\draw[thick, dashed] (5.35, -4.25) to [out=-75, in=115] (5.6, -5);
\draw[thick] (5.6, -5) to [out=-75, in=90] (5.99, -6.5) to (6,-8);
\draw[thick, dashed] (7.65, -4.25) to [out=-115, in=75] (7.3, -5);
\draw[thick] (7.3, -5) to [out=-115, in=90] (6.77, -6.5) to (6.76,-8);
\draw[thick] (6,-8) to [out =-15, in=-165] (6.76, -8);
\draw[thick] (6,-7) to [out =-15, in=-165] (6.76, -7);
\draw[thick] (6, -7.5) to [out=-25, in =-180] (6.5, -7.8) to [out =0, in =-165] (6.76, -7.6);
\draw[thick] (5.99, -6.3) to [out=-15, in =-165] (6.77, -6.3);

\end{tikzpicture}
\caption{Two scenarios.}
\label{fig-W-W'}
\end{figure}

\noindent
The following records the properties of $W$ needed in order to apply Theorem \ref{thm-subdiffeo}.

\begin{cor}\label{cor-W}
Assume that 
\be\label{three-bounds-again}
m<\e\cdot \tfrac{\sigma}{32} \textrm{ and } m<R\e^3  
\textrm{ where }0<\e<\e_0:=\sqrt{\tfrac{2}{\pi C_1^2}}.  
\ee
\begin{enumerate}[I]
\item \label{cor-U}
We have 
\be 
W \subseteq M_1\,\cap \, M_2,
\ee
where $M_1$ is as in (\ref{defn-M_1})
and $M_2$ is as in (\ref{defn-M_2}).
\item \label{W'-e} 
The following estimate holds over $W$:
\be
\delta \le g\le (1+\e)^2 \delta.
\ee
\item \label{W'-lambda} 
The parameter
\be
\lambda:=\sup_{x,y\in W} |d_{(M',g)}(x,y)-d_{(\mathbb{R}^3,\delta)}(x,y)|
\ee
satisfies 
\be
\lambda < \lambda_{R, \e}=24 R\e.
\ee
\item \label{M_1-W}
We have 
\be
M_1\setminus W\subseteq \bigg(B_\delta (0,R)\setminus B_\delta (0, R-\lambda)\bigg)\cup \bigcup_i \left(B_\delta (p_i, \gamma_{i,\e})\setminus B_\delta (p_i, \delta_{i,R})\right),
\ee
where clearly
$B_\delta (p_i, \gamma_{i,\e})\setminus B_\delta (p_i, \delta_{i,R})
=\emptyset$ when $\gamma_{i,\e}\le \delta_{i,R}$ of (\ref{defn-delta}).
\item \label{M_2-W}
We have  
\be
M_2\setminus W= \bigg(B_\delta (0,R)\setminus B_\delta (0, R-\lambda)\bigg)\cup \bigcup_i B_\delta (p_i, \gamma_{i,\e}). 
\ee
\item \label{accpart} Suppose that 
\be\label{accthing}
\rho(P)\,\cap \, (R-32R\e, R+32R\e)=\emptyset
\ee
where $\rho(x)=|x|$ and $P=\{p_1,...,p_N\}$. Then the unions in parts \eqref{M_1-W} and \eqref{M_2-W} of this corollary are disjoint. 
\end{enumerate}
\end{cor}

\begin{proof}
Let $x\in W$. Since by definition $W\subseteq W'$ and since $W'\subseteq M'$ by Lemma \ref{lemma2:oct} we have that $W\subseteq M'$. Furthermore, we see from 
\be 
d_{(M',g)}(0,x)<|x|+\lambda<R
\ee 
that $x\in B_g(0, R)\,\cap \, M'= M_1$. Observing that $x\in  B_\delta(0, R-\lambda)\subset M_2$ completes the proof of part\ref{cor-U}. Parts \ref{W'-e} and \ref{W'-lambda} are now immediate from Lemmas \ref{lemma2:oct} and \ref{lemma1:oct}, respectively. Note that hypotheses \eqref{three-bounds-again} imply $m<\sigma(20C_1)$ so that Lemma \ref{locatingM1} applies. Consequently,  we have parts \ref{M_1-W} and \ref{M_2-W} of our corollary. It remains to assume \eqref{accthing} and argue that the unions in parts \ref{M_1-W} and \ref{M_2-W} are disjoint: 
\be\label{acc} 
B_\delta (0,R)\setminus B_\delta (0, R-\lambda)\subset 
\mathbb{R}^3 \setminus \bigcup_i B_\delta(p_i, \gamma_{i,\e}).
\ee
Assumption \eqref{accthing}, and Lemma \ref{jan-add}, imply 
\be \label{acc-2}
||p_i|-R|\ge 32R\e > \lambda_{R,\e}+8R\e^2 > \lambda + \tfrac{8m}{\e} \ge \lambda+\gamma_{i,\e} >\gamma_{i,\e}.
\ee
for all $1\le i\le n$. In particular, we have 
\be
|p_i|< R-\lambda-\gamma_{i,\e} \textrm{ or } |p_i|> R+\gamma_{i,\e}
\ee
for all $1\le i\le n$ and the inclusion in \eqref{acc} is now immediate. 
\end{proof}

\subsection{Estimating Volumes}

To complete the estimation of the intrinsic flat distance between $(M_1,g)$ and $(M_2,\delta)$ we need estimates on 
volumes and areas in the hypothesis of Theorem~\ref{thm-subdiffeo}.
This is done in Lemma~\ref{lemma3a} and Lemma~\ref{lemma3b}.

\begin{lemma}\label{lemma3a}
Let $\e_0:=\sqrt{\tfrac{2}{\pi C_1^2}}\in (0,1)$. There exists a universal constant $C'>0$ such that for all $R>0$, all $\e \in (0,\e_0)$, and all geometrostatic manifolds $(M,g)$ with  
\be
m_{ADM}(M') = m < R \e^3,\ \ \ m<\e\cdot \tfrac{\sigma}{32}
\ee 
the region $W$ of \eqref{U} and Corollary \ref{cor-W} satisfies
\begin{align}
&\mathrm{Vol}_\delta(W)\le \mathrm{Vol}_g(W)\le C'R^3\\
&\mathrm{Vol}_\delta(\partial W)\le \mathrm{Vol}_g(\partial W)\le C' R^2. 
\end{align}
\end{lemma}

\begin{proof}
Fix $\e$ with $0<\e<\e_0$. 
The first estimate follows from the fact that 
\begin{eqnarray}  
\mathrm{Vol}_g(W)
&\le& (1+\e)^3 \mathrm{Vol}_\delta\left( W \right)
\textrm{ by Lemma~\ref{lemma1:oct}}\\
&\le& (1+\e)^3 \mathrm{Vol}_\delta \left( B_\delta(0, R-\lambda) \right)
\textrm{ by $W\subset B_\delta(0,R-\lambda)$,}\\
&\le& \frac{4\pi}{3}(1+\e_0)^3 (R-\lambda)^3\le C'R^3.
\end{eqnarray}
for some universal constant $C'>0$.
To prove the second estimate we have:
\be \begin{aligned}
\mathrm{Vol}_g(\partial W)
\le&(1+\e)^2\mathrm{Vol}_\delta(\partial W) \textrm{ by Lemma~\ref{lemma1:oct},}\\
\le&4\pi (1+\e)^2 \left(R^2+\sum \gamma_{i,\e}^2 \right)
\textrm{ by (\ref{U}),}\\
\le&4\pi (1+\e)^2 \left(R^2+96m^2/\e^2\right)
\textrm{ by Lemma~\ref{jan-add}},
\\
\le& 4\pi (1+\e_0)^2\left(R^2+96R^2\e_0^4\right)\\
\le &\ C'R^2 
\end{aligned}
\ee 
for some universal constant $C'>0$. 
\end{proof}

\begin{lemma}\label{lemma3b}
Let $\e_0:=\sqrt{\tfrac{2}{\pi C_1^2}}\in (0,1)$. There exists a universal constant $C''>0$ such that for all $R>0$, all $\e \in (0,\e_0)$, and all geometrostatic manifolds $(M,g)$ with  
\be\label{lemma3bm}
m_{ADM}(M') = m < R \e^3,\ \ \ m<\e\cdot \tfrac{\sigma}{32},\ \ \ \rho(P)\,\cap \, (R-32R\e, R+32R\e)=\emptyset,
\ee 
the regions $M_1$ of(\ref{defn-M_1}), $M_2$ of (\ref{defn-M_2}), and $W$ of \eqref{U} and Corollary \ref{cor-W} satisfy
\begin{eqnarray}
\label{V-1-R-e}&\mathrm{Vol}_g(M_1\setminus W)\le C''R^3 \e,\\ 
\label{V-2-R-e}&\mathrm{Vol}_\delta(M_2\setminus W)\le C''R^3 \e.
\end{eqnarray}
\end{lemma}

\begin{proof} Fix $\e$ with $0<\e<\e_0$. 
We start by proving \eqref{V-2-R-e}, as it is easier to establish.
Lemma ~\ref{jan-add} and parts \eqref{W'-lambda} and \eqref{M_2-W} of Corollary \ref{cor-W} imply that 
\begin{align}
\,\,\qquad\mathrm{Vol}_\delta (M_2\setminus W) &=
\mathrm{Vol}_\delta(B_\delta (0,R)\setminus B_\delta (0, R-\lambda))
 + \sum_i \mathrm{Vol}_\delta(B_\delta(p_i,\gamma_{i,\e}))\\
&= \tfrac{4}{3} \pi R^3 - \tfrac{4}{3} \pi (R-\lambda)^3
+ \sum_i \tfrac{4}{3} \pi (\gamma_{i,\e})^3\\
&\le 
\tfrac{4\pi}{3}\left(3R^2\lambda +\lambda^3+768 m^3/\e^3\right)\\
& \le \tfrac{4\pi}{3}\left(3R^2(24R\e) +(24R\e)^3+ 768 R^3 \e^6\right)\\
& \le C'' R^3\e 
\end{align}
for some universal constant $C''$. 

The inequality \eqref{V-1-R-e} is far more difficult to prove.
By part \eqref{M_1-W} of Proposition \ref{cor-W} (see Figure~\ref{fig-W-W'})
we have
\be \label{lemma3b-V}
\mathrm{Vol}_g(M_1\setminus W)
\le  
\mathrm{Vol}_g(B_\delta (0,R)\setminus B_\delta (0, R-\lambda))
+ \sum_i V_i
\ee
where
\be
V_i:=\mathrm{Vol}_g\left(B_\delta (p_i, \gamma_{i,\e})\setminus B_\delta (p_i, \delta_{i,R})\right).
\ee

We see from part \eqref{accpart} of Corollary \ref{cor-W} that \eqref{minithmM'} and \eqref{minithmM''} apply, giving us 
\begin{eqnarray}
\quad\mathrm{Vol}_g
\left(B_\delta (0,R)\setminus B_\delta (0, R-\lambda)\right)
&\le&
(1+\e)^3\,\mathrm{Vol}_\delta
\left(B_\delta (0,R)\setminus B_\delta (0, R-\lambda)\right)\\
&\le & (1+\varepsilon)^3 
\left( \tfrac{4}{3} \pi R^3 -  \tfrac{4}{3} \pi (R-\lambda)^3 \right) \\
&\le & \tfrac{4}{3}\pi (1+\varepsilon_0)^3 \left(3R^2\lambda+\lambda^3\right).
\end{eqnarray}
Combining this with (\ref{lemma3b-V}) and part \eqref{W'-lambda} of Corollary \ref{cor-W} we have
\be \label{lemma3b-V'}
\mathrm{Vol}_g(M_1\setminus W)
\le  
\tfrac{4}{3}\pi (1+\varepsilon_0)^3\left(3R^2(24R\e) +(24R\e)^3\right)+ \sum_i V_i.
\ee

Next we estimate
each term, $V_i$, in the sum. 
Let $x\in B_\delta (p_i,\gamma_{i,\e})\setminus B_\delta (p_i,\delta_{i,R})$ and let $j\neq i$. The definition of $\sigma$ in \eqref{separation}, Lemma~\ref{jan-add}
and our hypothesis (\ref{lemma3bm}) imply
\be
|x-p_j| \ge |p_i-p_j| - |x-p_i| > \sigma - \gamma_{i,\e} \ge \sigma- (8m/\e) > 24m/\e.
\ee
Combining this with $\sum_j (\alpha_j +\beta_j) \le m$ in (\ref{ADM})
we have
\be 
1+\sum_{j\neq i}\frac{\alpha_j}{|x-p_j|}<1+\frac{\e}{24}<2
\quad\textrm{ and }\quad
1+\sum_{j\neq i}\frac{\beta_j}{|x-p_j|}<1+\frac{\e}{24}<2.
\ee
In particular, it follows that 
\be 
V_i
\le   \int_{\delta_{i,R}}^{\gamma_{i,\e}} 
\left(2+\tfrac{\alpha_i}{r}\right)^3\left(2+\tfrac{\beta_i}{r}\right)^3\pi r^2\,dr.  
\ee

After expanding the integrand in terms of powers of $r$ and integrating individual terms we obtain:
\begin{align}
V_i
\le& \,\, 
(64/3)  \pi \gamma_{i,\e}^3 + 48\pi (\alpha_i+\beta_i)\gamma_{i,\e}^2 \\
&+48\pi  \left(\alpha_i\beta_i +(\alpha_i+\beta_i)^2\right)\gamma_{i,\e}\\
&-8\pi(\alpha_i+\beta_i)\left((\alpha_i+\beta_i)^2+6\alpha_i\beta_i\right)\ln(\delta_{i,R}/\gamma_{i,\e})\\
&+12\pi \left(\alpha_i\beta_i+(\alpha_i+\beta_i)^2\right) (\alpha_i\beta_i)/\delta_{i,R}
\\
&+6\pi(\alpha_i+\beta_i)\left((\alpha_i\beta_i)/\delta_{i,R}\right)^2
+\pi \left(( \alpha_i\beta_i)/(\delta_{i,R})\right)^3
\end{align}

By the definition of $\delta_{i,R}$ in \eqref{defn-delta} 
and Lemma~\ref{jan-add}, we have 
\be
\frac{\delta_{i,R}}{\gamma_{i,\e}} \ge 
\frac{\e(\alpha_i+\beta_i) \exp(-R/(\alpha_i+\beta_i))}{8m}.
\ee
Thus 
\be
-\ln\left(\frac{\delta_{i,R}}{\gamma_{i,\e}}\right)
\le 
-\ln\left(\frac{\e(\alpha_i+\beta_i)}{8m}\right)
+ \frac{R}{(\alpha_i+\beta_i)}
\ee
Since $\sup_{x>0}\left(-x \ln(x)\right)=e^{-1}<1/2$ we have $-\ln(x)< 1/(2x)$ for positive $x$.  So
\be
-\ln\left(\frac{\e(\alpha_i+\beta_i)}{8m}\right) \le
\frac{4m}{\e(\alpha_i+\beta_i)}.
\ee
Combining this with the above and multiplying by $(\alpha_i+\beta_i)$ we get,
\be
- (\alpha_i+\beta_i)\ln\left(\frac{\delta_{i,R}}{\gamma_{i,\e}}\right)
\,\,\le\,\, \frac{4m}{\e}+ R.
\ee
Furthermore by (\ref{defn-delta})
\be 
\frac{\alpha_i\beta_i}{\delta_{i,R}}\le 4C_1(\alpha_i+\beta_i).
\ee
Together with estimates such as $4\alpha_i\beta_i\le (\alpha_i+\beta_i)^2$ we have
\begin{align}
\sum_i V_i
\le& \,\, 
(64/3)  \pi \sum_i\gamma_{i,\e}^3 + 48\pi \sum_i (\alpha_i+\beta_i)\gamma_{i,\e}^2 \\
&+48\pi  \sum_i [(\alpha_i+\beta_i)^2/4 +(\alpha_i+\beta_i)^2]\gamma_{i,\e}\\
&+8\pi\sum_i\left((\alpha_i+\beta_i)^2+6(\alpha_i+\beta_i)^2/4\right)
\left(\frac{4m}{\e}+ R\right)
\\
&+12\pi \sum_i((\alpha_i+\beta_i)^2/4+(\alpha_i+\beta_i)^2) 4C_1(\alpha_i+\beta_i)
\\
&+6\pi\sum_i(\alpha_i+\beta_i)\left(4C_1(\alpha_i+\beta_i)\right)^2
+\pi \sum_i \left(4C_1(\alpha_i+\beta_i)\right)^3
\end{align}

By Lemma \ref{jan-add} we have 
\begin{eqnarray}
\sum_i \gamma_{i,\e}^3 \qquad\quad\,\,&<& 768 m^3/\e^3 \\
\sum_i (\alpha_i+\beta_i)\gamma_{i,\e}^2 \,&<& 96m^3/\e^2 \\
\sum_i (\alpha_i+\beta_i)^2\gamma_{i,\e}&\le&\left(\sum_i  (\alpha_i+\beta_i)^4\right)^{1/2}\left(\sum_i \gamma_{i,\e}^2\right)^{1/2}\\
&<& \left({\sum_i} (\alpha_i+\beta_i)\right)^2\sqrt{96m^2/\e^2}
\,\,\,< \,\,\, 10m^3/\e, 
\end{eqnarray}
with the last set of estimates holding due to  
\be
\sum_i(\alpha_i+\beta_i)^4 < \left( \sum_i(\alpha_i+\beta_i)\right)^4\le m^4.
\ee
Likewise, $\sum_i(\alpha_i+\beta_i)^3<m^3$ and thus
\begin{align}
\sum_i V_i
\le& \,\, 
(64/3)  \pi\cdot 768 m^3/\e^3 + 48\pi\cdot  96m^3/\e^2 +60\pi\cdot 10m^3/\e \\
&+20\pi m^2
\left(4m/\e + R\right)
+60\pi C_1 m^3 
+96\pi C_1^2 m^3
+64 \pi C_1^3 m^3\\
\le&  \,\, 
(64/3)  \pi\cdot 768 R^3\e^6 + 48\pi\cdot 96R^3\e^7 +60\pi\cdot 10R^3\e^8 \\
&+20\pi R^3\e^6
\left(4\e^2 + 1\right)
+60\pi C_1 R^3\e^9 
+96\pi C_1^2 R^3\e^9
+64 \pi C_1^3 R^3\e^9\\
\le &\ C''R^3 \e^6, 
\end{align}
for some universal constant $C''$.
Combining this with (\ref{lemma3b-V'}) we finally have our first inequality.
\end{proof}

\subsection{The proof of Proposition \ref{main-prop}}
We now prove Proposition~\ref{main-prop} bounding the intrinsic flat distance between 
\be 
M_1=\{x\in M'\big{|}\, d_{(M',g)}(0,x)<R\}=B_g(0,R)\subseteq M'
\ee
endowed with the distance $d_{(M',g)}$
and
\be 
 M_2=\{x\in \mathbb{R}^3\big{|} \, |x|<R\}= B_\delta(0,R)
\ee 
endowed with the distance $d_{(\mathbb{R}^3, d_\delta)}(x,y)=|x-y|$.
Note that the following estimate on the diameters of these regions
\be
\,\max\{\diam_{g}(M_1), \diam_\delta(M_2)\}\le D=2R.
\ee

\begin{proof}[Proof of Proposition \ref{main-prop}]
We prove this proposition using the method of Lakzian and the first author to estimate the intrinsic flat distance (cf. Theorem~\ref{thm-subdiffeo}):
\be \label{need-this-LS}
\begin{aligned}
d_{\mathcal{F}}(M_1, M_2)\le &(2\bar h+a) \left(\text{Vol}_g(W)+\text{Vol}_\delta(W)+\text{Vol}_g(\partial W)+\text{Vol}_\delta(\partial W)\right)\\
&+\text{Vol}_g(M_1\setminus W)+\text{Vol}_\delta (M_2\setminus W)
\end{aligned}
\ee
taking $W$ as defined in (\ref{U}) and addressed in detail in Corollary~\ref{cor-W}. 

Let $\e_0:=\sqrt{\tfrac{2}{\pi C_1^2}}\in (0,1)$ and let $0<\e<\e_0$. 
Define $a$ as in (\ref{thm-subdiffeo-3}):
\be
a=\frac{\arccos[(1+\e)^{-1}]}{\pi}\cdot D=\frac{\arccos[(1+\e)^{-1}]}{\pi}\cdot 2R.
\ee
It follows from the L'H\^{o}pital's Rule that 
\be 
\e \mapsto \e^{-1/2}\cdot \arccos\left((1+\e)^{-1}\right)
\ee is a positive bounded function of $\e\in [0,\infty)$. Thus there is a universal constant $C_a$ such that 
\be\label{a-estimate}
0<a\le C_a R \sqrt{\e}. 
\ee
In addition, by part \eqref{W'-lambda} of Corollary \ref{cor-W} we have 
that $\bar{h}$ of \eqref{thm-subdiffeo-5} 
satisfies
\begin{eqnarray}
\bar{h}&=& \,\max \left\{\sqrt{2\lambda D}, D\sqrt{\e^2+2\e} \right\}\\
&\le & \,\max \left\{ \sqrt{2 (24 R \e) 2 R}, 2R \sqrt{3\e}\right\}  \le 10 R \sqrt{\e}.
\end{eqnarray}
Overall, we see that
\be
0<2\bar h+a \le C_2 R \sqrt{\e} 
\ee
for some universal constant $C_2$. Substituting our bound on $2\bar h+a$ into (\ref{need-this-LS}) along with the volume estimates from Lemmas~\ref{lemma3a} and~\ref{lemma3b} we have,
\be
d_{\mathcal{F}}(M_1, M_2) \le 
C_2 R \sqrt{\e}\cdot\left( 2C' R^3+ 2 C' R^2 \right)+ 2C''R^3\e,
\ee
which in turn implies \eqref{flatdistance-est}.
Similarly, using $D=2R$ we have
\be
d_{D\mathcal{F}}(M_1, M_2) \le 
\tfrac{1}{2R}\cdot C_2 R \sqrt{\e}\cdot \left( 2C' R^3+ 2\,(2R)\, C' R^2 \right)+ 2C''R^3\e,
\ee
which implies \eqref{D-flatdistance-est}.
\end{proof}

\subsection{The proof of Theorem \ref{Almost-Rigidity}}\label{proof-mainthm}

\begin{proof}
By assumption there is some 
$R_0\ge 0$ such that for almost every $R>R_0$, $R$ is not an accumulation point of $\rho(\cup_k P_k)$. Fix one such value of $R$. We show that for all $\bar{\epsilon}>0$ there is $K=K(\bar{\epsilon})\in \mathbb{N}$ such that 
\be\label{finalconv}
d_{\mathcal{F}}(B_{g_k}(0,R), B_\delta(0,R)) < \bar{\epsilon} \text{\ \ and\ \ }
d_{D\mathcal{F}}(B_{g_k}(0,R), B_\delta(0,R)) < \bar{\epsilon}
\ee
for all $k\ge K$.

For our fixed value of $R$ there exist $R'>0$ and $k_0\in \mathbb{N}$ such that 
\be \label{choose-N-1}
\rho(P_k) \,\cap \, (R-R', R+R') = \emptyset
\ee
for all $k\ge k_0$.
Take $0<\e<\e_0$ sufficiently small so that 
\begin{align}\label{choose-e-1}
&32R\e < R',\\
&C_{\mathcal{F}}' R^4 \sqrt{\e} + C_{\mathcal{F}}'' R^3 \sqrt{\e}<\bar{\epsilon},\\
&C_{D\mathcal{F}} R^3 \sqrt{\e}<\bar{\epsilon}
\end{align}
for constants $\e_0$, $C_{\mathcal{F}}'$, $C_{\mathcal{F}}''$ and $C_{D\mathcal{F}}$ of Proposition \ref{main-prop}. Finally, by assumptions \eqref{hyp-thm1.5} we know that there exists a $K\ge k_0$ such that
\be \label{choose-N-e}
m_{ADM}(M'_k) < R\e^3
\textrm{\ \ and\ \ }
\frac{m_{ADM}(M'_k)}{\sigma(M_k)} < \frac{\e}{32}
\ee
for all $k\ge K$. The hypotheses of Proposition \ref{main-prop} are now satisfied and as a result we obtain \eqref{finalconv}. This completes our proof.  
\end{proof}

        

\appendix

\section{Inverting a geometrostatic manifold}



Recall that the Riemannian Schwarzschild manifold of mass $m_1=m$
is a geometrostatic manifold with a single point $p_1=0$ and that this manifold has an isometry which interchanges the two ends:
\be
F: \mathbb{R}^3\setminus \{0\} \to \mathbb{R}^3\setminus \{0\}
\textrm{ defined by } F(y) = \left(\frac{m}{2}\right)^2\frac{y}{\,|y|^2}
\textrm{ with } F^{-1}(x)= \left(\frac{m}{2}\right)^2\frac{x}{\,|x|^2}.
\ee
That is the pullback metric of $g_{\mathrm{Sch}}$ as in (\ref{Sch}) is
\begin{eqnarray}
(F^*g_{\mathrm{Sch}})_y&=&\left(1+\frac{m}{2|F(y)|}\right)^4(F^*\delta)_y
=\left(1+\frac{2|y|}{m}\right)^4
\left(\frac{m}{2}\right)^4\frac{1}{\,|y|^4}\,\, \delta_y\\
&=&\left(\left(\frac{m}{2}\right)\frac{1}{|y|}+1\right)^4\delta_y
\,\,\,=\,\,\, (g_{\mathrm{Sch}})_ y.
\end{eqnarray}

Similarly, if we apply an inversion to any geometrostatic manifold
taking the end at infinity to an end at the origin and
taking one of the other ends to the end at infinity we obtain an
isometric geometrostatic manifold:

\begin{thm}\label{inversion}
Let $(X,g_X)$ be the geometrostatic manifold: 
\be
X=\mathbb{R}^n \setminus \{x_1,...,x_n\}
\textrm{ with }
g_X= \left(1+\sum_{i=1}^n \frac{\alpha_{X,i}}{|x-x_i|}\right)^2
\left(1+\sum_{i=1}^n \frac{\beta_{X,i}}{|x-x_i|}\right)^2 \delta_x
\ee
Let
\be
F(y)= \alpha_n\beta_n \frac{y}{\,|y|^2} + x_n
\textrm{ so that } F^{-1}(x)= \alpha_n\beta_n \frac{x-x_n}{|x-x_n|^2}.
\ee
Let $y_0=0$ and
\be
y_j=F^{-1}(x_j)= \alpha_n\beta_n \frac{x_j-x_n}{|x_j-x_n|^2}= \alpha_n\beta_n \frac{x_j-x_n}{r_{jn}^2} \quad
\textrm{ for } j =1,..,n-1.
\ee
Let $(Y, g_Y)$ be the geometrostatic manifold
\be
Y=\mathbb{R}^n\setminus\{y_0,...,y_{n-1}\} 
\textrm{ with }
g_Y=\left(1+\sum_{i=0}^{n-1} \frac{\alpha_{Y,i}}{|y-y_i|}\right)^2
\left(1+\sum_{i=0}^{n-1} \frac{\beta_{Y,i}}{|y-y_i|}\right)^2 \delta_y
\ee
where
\begin{eqnarray}
\alpha_{Y,0} &=& \alpha_n \\
\beta_{Y,0}&=& \beta_n\\
\alpha_{Y,j}&=& \frac{\beta_{X,j}\alpha_n}{r_{j,n}}\,\,=\,\, \frac{\beta_{X,j} |y_j|}{\beta_n}
 \\
\beta_{Y,j}& =&  \frac{\alpha_{X,j}\beta_n}{r_{j,n}}\,\,=\,\,\frac{\alpha_{X,j} |y_j|}{\alpha_n}
\end{eqnarray}
Then $F: Y \to X$ is an isometry which
maps the end at infinity for $Y$ to the end at $x_n$ for $X$
and the end at $0$ for $Y$ to
the end at infinity for $X$, and the end at $y_j$ for $Y$ to the 
end at $x_j$ for $X$.   
\end{thm}

The proof of this theorem follows a method of Misner in \cite{Misner-geo}.   We include it for completeness of exposition.

\begin{proof}
We need only prove $g_Y=F^*g_X$.

First observe that 
\be
\left(F^*\delta\right)_y=\left(\alpha_n \beta_n\right)^2\frac{1}{\,|y|^4}\,\, \delta_y.
\ee
So
\begin{eqnarray*}
\left(F^*g_X\right)_y &=& 
\left( 1 + \sum_{j=1}^n \frac{\alpha_{X,j}}{|F(y)-x_j|}\right)^2
\left( 1 + \sum_{j=1}^n \frac{\beta_{X,j}}{|F(y)-x_j|}\right)^2 \left(F^*\delta\right)_y\\
&=& 
\left( \left( \frac{\beta_n}{|y|}\right)
       \left(1 + \sum_{j=1}^n \frac{\alpha_{X,j}}{|F(y)-x_j|}\right)\right)^2
\left( \left( \frac{\alpha_n}{|y|}\right)
       \left(1 + \sum_{j=1}^n \frac{\beta_{X,j}}{|F(y)-x_j|}\right)\right)^2
 \,\,\,\delta_y.
\end{eqnarray*} 
We need only prove 
\begin{eqnarray}
 \left( \frac{\beta_n}{|y|}\right)
       \left(1 + \sum_{j=1}^n \frac{\alpha_{X,j}}{|F(y)-x_j|}\right)
&=&
\left(1 + \sum_{j=0}^{n-1} \frac{\beta_{Y,j}}{|y-y_j|}\right)\\
 \left( \frac{\alpha_n}{|y|}\right)
       \left(1 + \sum_{j=1}^n \frac{\beta_{X,j}}{|F(y)-x_j|}\right)
&=&
\left(1 + \sum_{j=0}^{n-1} \frac{\alpha_{Y,j}}{|y-y_j|}\right).
\end{eqnarray}
First observe that
\begin{eqnarray}
\frac{\beta_n}{|y|}&=& \frac{\beta_n}{|y-y_0|} = \frac{\beta_{Y,0}}{|y-y_0|} 
\textrm{ by the choice of $y_0=0$ and $\beta_{Y,0}=\beta_n$,}\\
\frac{\alpha_n}{|y|}&=& \frac{\alpha_n}{|y-y_0|} = \frac{\alpha_{Y,0}}{|y-y_0|} 
\textrm{ by the choice of $y_0=0$ and $\alpha_{Y,0}=\alpha_n$.}
\end{eqnarray}
Since $\alpha_{X,n}=\alpha_n$ and $\beta_{X,n}=\beta_n$ we have,
\begin{eqnarray}
\qquad
\frac{\beta_n}{|y|}
 \left(\frac{\alpha_{X,n}}{|F(y)-x_n|}\right)
 &=& \frac{\beta_n}{|y|}
 \left(\frac{\alpha_{n}}{|\alpha_n\beta_n y/|y|^2 + x_n-x_n|}\right)
=\frac{\beta_n \alpha_{n}}{|\alpha_n\beta_n y/|y||}= 1, \\
\quad \frac{\alpha_n}{|y|}
 \left(\frac{\beta_{X,n}}{|F(y)-x_n|}\right)
 &=& \frac{\alpha_n}{|y|}
 \left(\frac{\beta_{n}}{|\alpha_n\beta_n y/|y|^2 + x_n-x_n|}\right)
=\frac{\alpha_n \beta_{n}}{|\alpha_n\beta_n y/|y||}= 1. 
\end{eqnarray}
So we need only prove for $j=1,..., n-1$ we have
\begin{eqnarray}
 \left( \frac{\beta_n}{|y|}\right)
       \left( \frac{\alpha_{X,j}}{|F(y)-x_j|}\right)
&=&
\left(\frac{\beta_{Y,j}}{|y-y_j|}\right)\\
 \left( \frac{\alpha_n}{|y|}\right)
       \left( \frac{\beta_{X,j}}{|F(y)-x_j|}\right)
&=&
\left( \frac{\alpha_{Y,j}}{|y-y_j|}\right).
\end{eqnarray}
Now
\begin{eqnarray}
|F(y)-x_j|^2&=& |F(y)-F(y_j)|^2\\
&=&\left| \alpha_n\beta_n \frac{y}{\,|y|^2} +x_n -\alpha_n\beta_n \frac{y_j}{\,|y_j|^2}-x_n\right|^2\\
&=&(\alpha_n\beta_n)^2 \left|\frac{y}{\,|y|^2}  -\frac{y_j}{\,|y_j|^2}\right|^2\\
&=&\frac{(\alpha_n\beta_n)^2}{|y|^4|y_j|^4} \left| |y_j|^2y-|y|^2y_j\right|^2\\
&=&\frac{(\alpha_n\beta_n)^2}{|y|^4|y_j|^4} \left( |y_j|^4|y|^2+|y|^4|y_j|^2
- 2 |y_j|^2|y|^2 \, y\cdot y_j \right)\\
&=&\frac{(\alpha_n\beta_n)^2}{|y|^2|y_j|^2}  \left( |y_j|^2+|y|^2- 2  \, y\cdot y_j \right)\\
&=& \frac{(\alpha_n\beta_n)^2}{|y|^2|y_j|^2}  \,|y-y_j|^2.
\end{eqnarray}
Applying this to $y=y_k$ we have
\be \label{ykyj}
r_{k,j}^2=|x_k-x_j|^2 =  \frac{(\alpha_n\beta_n)^2}{|y_k|^2|y_j|^2}  \,|y_k-y_j|^2.
\ee
Applying it more generally for $j \neq n$ we have
\begin{eqnarray}
 \left( \frac{\beta_n}{|y|}\right)
       \left( \frac{\alpha_{X,j}}{|F(y)-x_j|}\right)
&=& \left( \frac{\beta_n}{|y|}\right)
       \left( \frac{\alpha_{X,j}\,|y|\,|y_j|}{\alpha_n \beta_n \,|y-y_j|}\right)
 =\frac{\beta_{Y,j}}{|y-y_j|}\\
 \left( \frac{\alpha_n}{|y|}\right)
       \left( \frac{\beta_{X,j}}{|F(y)-x_j|}\right)
&=& \left( \frac{\alpha_n}{|y|}\right)
       \left( \frac{\beta_{X,j}|y|\,|y_j|}{\alpha_n \beta_n \,|y-y_j|}\right)
 =\frac{\alpha_{Y,j}}{|y-y_j|}.
\end{eqnarray}
\end{proof}

\begin{cor} \label{mADM-Y}
Since $F$ is an isometry which maps the end at infinity for $Y$ to the end at $x_n$ for $X$ these ends have the same ADM mass:
\be
m_{ADM}(Y)=m_{X,n}.
\ee
Since $F$ is an isometry which maps the end at $y_0$ for $Y$ to the end at $\infty$ for $X$, these ends have the same ADM mass:
\be
m_{Y,0}=m_{ADM}(X).
\ee
\end{cor}

\begin{proof}
We also prove these equalities using (\ref{ADM}) and (\ref{m_i}) thus
rederiving these equations that were stated without
proof in Brill-Lindquist.  We have by (\ref{ADM}) that the ADM mass of the end at infinity for $Y$
satisfies:
\begin{eqnarray}
\,\,\,m_{ADM}(Y)&=& \sum_{j=0}^{n-1} (\alpha_{Y,j}+\beta_{Y,j})=
\alpha_{X,n} + \beta_{X,n} +
\sum_{j=1}^{n-1}\left(\frac{\beta_{X,j}\alpha_n}{r_{j,n}} +  \frac{\alpha_{X,j}\beta_n}{r_{j,n}}\right)\\
&=&\alpha_{X,n} + \beta_{X,n} +
\sum_{j=1}^{n-1}\left(\frac{\beta_{X,j}\alpha_{X,n}+\alpha_{X,j}\beta_{X,n}}{r_{j,n}}\right)\,\,=\,\, m_{X,n} \,\,\, \textrm{ by (\ref{m_i})}
\end{eqnarray}
where $m_{X,n}$ is the ADM mass of the end at $x_n$ for $X$.  This is also
a consequence of the fact that $F$ is an isometry which maps the end at infinity for $Y$ to the end at $x_n$ for $X$. We have by (\ref{m_i}) that the ADM mass of the end at $y_0$ for $Y$
satisfies
\begin{eqnarray} 
\,\,\,m_{Y,0}&=&
\alpha_{Y,0}+\beta_{Y,0}
+\sum_{j=1}^{n-1} \frac{(\beta_{Y,0} \,\alpha_{Y,j} +\alpha_{Y,0}\,\beta_{Y,j})}
{|y_0-y_j|}\\
&=&
\alpha_n+\beta_n
+\sum_{j=1}^{n-1} \frac{\,\,\beta_n (\beta_{X,j}\,\alpha_n/r_{j,n}) +
\alpha_n (\alpha_{X,j}\,\beta_n/r_{j,n})\,}
{(\alpha_n \beta_n/r_{j,n})}\\
&=&
\alpha_{X,n}+\beta_{X,n}
+\sum_{j=1}^{n-1} (\beta_{X,j}+ \alpha_{X,j}) \,= \, m_{ADM}(X)\,\,\,
\textrm{ by (\ref{ADM}) },
\end{eqnarray} 
where $m_{ADM}(X)$ is the ADM mass of the end at infinity for $X$.
This is also
a consequence of the fact that $F$ is an isometry which maps the end at $y_0$ for $Y$ to the end at $\infty$ for $X$. Recall that by (\ref{ykyj}) we have
\be
|y_k-y_j|= \frac{ r_{k,j} }{(\alpha_n\beta_n)}|y_k|\,|y_j|
= \frac{ r_{k,j} }{(\alpha_n\beta_n)}
\left(\frac{\alpha_n \beta_n}{r_{k,n}}\right)
\left(\frac{\alpha_n\beta_n}{r_{j,n}}\right) 
= \frac{\alpha_n\beta_n r_{k,j}}{r_{k,n}r_{j,n}}
\ee
We can combine this 
with (\ref{m_i}) to show that the 
ADM mass of the end at $y_k\neq y_0$ for $Y$
satisfies
\begin{eqnarray*} 
\qquad m_{Y,k}&=&
\alpha_{Y,k}+\beta_{Y,k}
+\frac{(\beta_{Y,k} \alpha_{Y,0} +\beta_{Y,0}\alpha_{Y,k})}
{|y_k-0|}
+\sum_{j\neq k, j=1}^{n-1} \frac{(\beta_{Y,k} \alpha_{Y,j} +\beta_{Y,j}\alpha_{Y,k})}
{|y_k-y_j|} \\
&=&
\frac{\beta_{X,k}\,\alpha_n}{r_{k,n}} + \frac{\alpha_{X,k}\,\beta_n}{r_{k,n}}
+\left( \frac{\alpha_{X,k}\,\beta_n\alpha_n}{r_{k,n}} 
+\frac{\beta_n\beta_{X,k}\,\alpha_n}{r_{k,n}}  \right)
\frac{r_{k,n}}{\alpha_n\beta_n}
\\
&&+\sum_{j\neq k, j=1}^{n-1} 
\frac{\,(\alpha_{X,k}\,\beta_n/r_{k,n}) (\beta_{X,j}\,\alpha_n/r_{j,n}) +
(\alpha_{X,j}\,\beta_n/r_{j,n})(\beta_{X,k}\,\alpha_n/r_{k,n})\,}
{(\alpha_n\beta_n r_{k,j})/(r_{k,n}r_{j,n})}
 \\
&=&
\frac{\,\beta_{X,k}\,\alpha_{X,n}+ \alpha_{X,k}\,\beta_{X,n}\,}{|x_k-x_n|}
+\left( \alpha_{X,k}+\beta_{X,k}  \right)
+\sum_{j\neq k, j=1}^{n-1} 
\frac{\,(\alpha_{X,k}\beta_{X,j} +
\alpha_{X,j}\beta_{X,k})\,}
{ r_{k,j}}
 \\
 &=&
\left( \alpha_{X,k}+\beta_{X,k}  \right)
+\sum_{j\neq k, j=1}^{n} 
\frac{\,(\alpha_{X,k}\beta_{X,j} +
\alpha_{X,j}\beta_{X,k})\,}
{ r_{k,j}} \,\,\,=\,\,\, m_{X,k} \,\,\,\textrm{ by (\ref{m_i})}
\end{eqnarray*}
where $m_{X,k}$ is the 
ADM mass of the end at $x_k\neq x_n$ for $X$.  This is
also a consequence of the fact that $F$ maps the end of $Y$
at $y_k$ to the end of $X$ at $x_k$ for $k=1,...,n-1$.
\end{proof}



 \ack 
 
The authors would like to thank Mingliang Cai, Qing Chen, Xiuxiong Chen, Piotr Chrusciel, Greg Galloway, Sen Hu and Jim Isenberg for organizing the July 2014 {\em Geometric Analysis and Relativity Conference} at the University of Science and Technology of China  at which this collaboration commenced.
We would also like to thank Greg Galloway, David Maxwell, Richard Schoen, and Daniel Pollack for organizing the Banff Geometric Analysis and General Relativity
Workshop in July 2016 during which the final version of this work was formulated.

Prof. Sormani's research is funded in part by a PSC CUNY Research Grant and by NSF DMS 1309360.
Prof. Stavrov's research is funded in part by travel grants from Lewis \& Clark College.


\frenchspacing
\bibliographystyle{cpam}

\end{document}